\documentclass[a4paper,12pt]{amsart}

\usepackage{graphics,graphicx}
\usepackage{tikz}

\usepackage[utf8]{inputenc}
\usepackage[T1]{fontenc}
\usepackage{amssymb}

\usepackage[sb]{libertine}
\usepackage[frenchmath]{libertinust1math}
\usepackage{calrsfs}

\newtheorem{mthm}{Theorem}

\newtheorem{thm}{Theorem}[section]
\newtheorem{cor}[thm]{Corollary}
\newtheorem{lem}[thm]{Lemma}
\newtheorem{prop}[thm]{Proposition}

\theoremstyle{definition}
\newtheorem{defn}[thm]{Definition}

\theoremstyle{remark}
\newtheorem{rem}[thm]{Remark}
\newtheorem{exa}[thm]{Example}

\DeclareMathOperator{\Ric}{Ric}
\DeclareMathOperator{\Scal}{Scal}
\DeclareMathOperator{\sign}{sign}

\DeclareMathOperator{\vol}{vol}

\begin{document}

\title{On toric Hermitian ALF gravitational instantons}

\author{Olivier Biquard \and Paul Gauduchon}
\address{Sorbonne Université and Université de Paris, CNRS, IMJ-PRG, F-75005 Paris}
\email{olivier.biquard@sorbonne-universite.fr}
\address{École Polytechnique, CNRS, CMLS, F-91120 Palaiseau}
\email{paul.gauduchon@polytechnique.edu}


\begin{abstract}
  We give a classification of toric, Hermitian, Ricci flat, ALF Riemannian metrics in dimension 4, including metrics with conical singularities. The only smooth examples are on one hand the hyperKähler ALF metrics, on the other hand, the Kerr, Taub-NUT and Chen-Teo metrics. There are examples with conical singularities with infinitely many distinct topologies. We provide explicit formulas.
\end{abstract}

\maketitle

More than ten years ago Chen and Teo discovered a new family of Riemannian AF gravitational instantons \cite{CheTeo11}, where AF (Asymptotically Flat) means that the end of the instanton is diffeomorphic to that of $\Bbb{R}^3\times S^1$. The family depends on one parameter (up to scale). Before that, only one family was known: the Kerr family (also a one parameter family up to scale), among which one finds the well-known Schwarzschild metric.

Our work started from the observation by Steffen Aksteiner and Lars Andersson \cite{AksAnd21} that the Chen-Teo instantons are `half algebraically special', or in more geometric terms that the selfdual Weyl tensor $W^+$ admits a double eigenvalue. This is equivalent to say that the metric is Hermitian, and more precisely is conformally Kähler. For a toric metric the problem can then be expressed in the setting of toric Kähler geometry, and the objects of interest are Bach flat extremal Kähler metrics with Poincaré behaviour at infinity, to be constructed from some given `moment polytope' in the real plane. This polytope completely characterizes the instanton up to scale, and we will calculate it for all classical examples.

This perspective gave some hope to construct other ALF gravitational instantons. ALF metrics are metrics which at infinity look locally like the product $\Bbb{R}^3\times S^1$ (see the precise definition in section \ref{sec:alf-metrics}). We give a complete classification of Hermitian, toric, Ricci flat, ALF metrics in terms of one real convex piecewise linear function on the real line. As a byproduct we obtain:

\begin{mthm}\label{thm:A}
  Suppose $(X^4,g)$ is a smooth Ricci flat, Hermitian, toric, ALF metric. If $g$ is not Kähler, then $g$ is a Kerr metric, a Chen-Teo metric, the Taub-NUT metric (with the orientation opposed to the hyperKähler orientation) or the Taub-bolt metric.
\end{mthm}
The Kähler examples are actually hyperKähler gravitational instantons and are well-known: one recovers only multi-Taub-NUT metrics. The Taub-NUT metric itself is special since it is hyperKähler for one orientation, and conformally Kähler for the other orientation, hence we have put it in the list of the theorem.

There is a curious duality between our classification and that of compact Hermitian Einstein metrics with positive Ricci which are not Kähler, see \cite{Leb97,CheLebWeb08}: these are known to be $\Bbb{C}P^2$ (the Fubini-Study metric with the reverse orientation), its blow up at 1 point (the Page metric) or 2 points (the Chen-LeBrun-Weber metric). The analogy with the Taub-NUT, Kerr and Chen-Teo metrics is obvious, and is confirmed by the fact that their polytopes have the same number of edges. But we have not found any direct relation between both classifications. The Taub-bolt instanton does not fit in this correspondence but can be considered as a degeneration of Chen-Teo instantons.

Of course theorem A leaves open the fascinating question of the classification of all gravitational instantons, or at least of Hermitian gravitational instantons, see \cite[conjecture 1]{AksAnd21}. Only the classification in the hyperKähler case is now known after the efforts of lots of mathematicians \cite{Kro89b,Min11,ChenChen15,ChenChen19,ChenChen21,SunZhang21}.

Theorem \ref{thm:A} is a classification theorem which does not give new examples of ALF instantons (nevertheless it gives an alternative construction of the Chen-Teo instantons). But we have a general classification result (theorem \ref{thm:classification}) in terms of convex piecewise linear functions, which also allows to understand metrics with conical singularities. These singularities occur around 2-spheres (divisors from the complex viewpoint) which are fixed by the torus action. Our results are quite different from the recent classification results in \cite{KunLuc21}: on one hand our techniques are not limited to the AF case, on the other hand by restricting to the Hermitian case we obtain a full description in the sense that we can understand all geometric properties (smoothness, cone angles, etc.) from our piecewise linear functions. In particular we obtain:

\begin{mthm}\label{thm:B}
  Suppose $(X^4,g)$ is a Hermitian, toric, Ricci flat, ALF instanton with conical singularities of cone angles $2\pi\alpha_i$ around the fixed divisors $D_i$. Let $x=D_j\cap D_{j'}$ be a fixed point of the action. Then for each $\alpha<\alpha_j+\alpha_{j'}$close to $\alpha_j+\alpha_{j'}$, the blowup of $X$ at the point $x$ admits a Hermitian, toric, Ricci flat, ALF instanton, with the same conical angles around the $D_i$'s, and a conical singularity of angle $2\pi\alpha$ around the additional divisor.
\end{mthm}
This theorem implies that by successive blowups one can construct ALF instantons with conical singularities on an infinite family of topologically distinct manifolds. Of course it follows from the condition on $\alpha$ that we need larger and larger angles.

As an illustration of the theorem, one can start from a Kerr metric and blow up a point: one obtains a family of metrics with cone angle slightly less than $4\pi$ on the exceptional divisor. It turns out that one can decrease the angle to the angle $2\pi$ to create a smooth instanton which is a Chen-Teo instanton. These metrics fit in the larger 5-parameter family of Chen and Teo in \cite{CheTeo15}, but we have not studied them systematically, see nevertheless remark \ref{rem:full-CT}.

It can be seen that the geometric process for this blow up, for example in the case $\alpha_j=\alpha_{j'}=1$, consists in gluing a ramified double cover of the Eguchi-Hanson metric on $T^*S^2$, hence with cone angle $4\pi$ around $S^2$.


\smallskip Since the problem of constructing toric Hermitian gravitational instantons can be reinterpreted in terms of Kähler metrics, we are in the general setting of the Yau-Tian-Donaldson conjecture on extremal Kähler metrics. This is solved in the toric case \cite{Don09,ChenCheng21a,ChenCheng21b,He19}, but the resolution does not seem to cover exactly our case because of the Poincaré behaviour at infinity. Even with solutions, it is a hard problem to determine when the extremal metric is Bach flat. But we avoid this problem, because the Ricci flat equation (becoming the Bach flat equation on the Kähler metric) is so strong that an ansatz exists for the solutions in terms of an axisymmetric harmonic function on $\Bbb{R}^3$. This was discovered recently by Paul Tod \cite{Tod20}, and we give another derivation of his ansatz in this paper. 

Applying the ansatz to our possible polytopes gives that all possible harmonic functions are classified in terms of a single convex piecewise linear real function on the real axis. We can then read the smoothness of the solution in terms of the usual theory of toric Kähler metrics (`Delzant polytopes') and proceed to a complete classification. 

\subsubsection*{Overview of the paper} In section \ref{sec:alf-metrics} we give a general definition of ALF metrics and introduce the Hermitian condition; in particular we find that a possible complex structure has only one possible behaviour at infinity (proposition \ref{prop:twistor-2-forms}). After recalling the toric machinery in section \ref{sec:toric-kahl-geom}, we give our derivation of the Tod ansatz in section \ref{sec:ansatz}. From the behaviour of the Kähler metric at the boundary of the polytope, we deduce in sections \ref{sec:boundary-conditions} and \ref{sec:the-solutions} that the generating function $U$ of Tod, which is an axisymmetric harmonic function on $\Bbb{R}^3$, is generated by a convex piecewise linear positive function on the real line (it is actually the Dirac density of $\Delta U$ along the real axis). We then study the metric just from the data of this convex piecewise linear function: the polytope is found in section \ref{sec:moment-map}, the regularity is studied in section \ref{sec:regularity} (theorem \ref{thm:classification}), where we also describe the classical examples from this point of view: these are the Kerr-Taub-bolt metrics \cite{GibPer80}, they all correspond to instantons with conical singularities, including on the divisors going to infinity. These ideas are applied to prove theorem \ref{thm:A} in section \ref{sec:class-smooth-herm}, which contains also a calculation of the polytopes of the Chen-Teo instantons. Section \ref{sec:blowing-up-cone} contains the construction of conical instantons by blowup, leading to theorem \ref{thm:B}.

\subsubsection*{Acknowledgements} The first author thanks Steffen Aksteiner and Lars Andersson for generously sharing their ideas on the Chen-Teo instantons, during a stay at the Institut Mittag-Leffler in Djursholm during the fall of 2019.

\section{ALF metrics and complex structures}
\label{sec:alf-metrics}

\begin{defn}
  A Riemannian manifold $(M^4,g)$ is ALF if
  \begin{itemize}
  \item it has one end diffeomorphic to $(A,+\infty)\times L$, where $L$ is $S^1\times S^2$, $S^3$, or a finite quotient;
  \item there is a triple $(\eta,T,\gamma)$ defined on $L$, where $\eta$ is a 1-form, $T$ a vector field such that $\eta(T)=1$ and $T\lrcorner d\eta=0$, and $\gamma$ is a $T$-invariant metric on the distribution $\ker \eta$;
  \item the transverse metric $\gamma$ has constant curvature $+1$;
  \item the metric $g$ has the behaviour
    \begin{equation}
     g = dr^2 + r^2 \gamma + \eta^2 + h, \text{ with } |\nabla^kh|=O(r^{-1-k}),\label{eq:1}
   \end{equation}
    where $\gamma$ is extended to the whole $TL$ be deciding that $\ker \gamma$ is generated by $T$, and the covariant derivative $\nabla$ and the norm are with respect to the asymptotic model $dr^2+r^2 \gamma+\eta^2$.
  \end{itemize}
\end{defn}
The AF case is when $L=S^2\times S^1$, so that the end has the topology of $(\Bbb{R}^3\setminus B(R))\times S^1$ for some large $R$.

The hypothesis tells us that locally near a point of $L$ the metric is $dr^2+r^2g_{S^2}+dt^2+O(r^{-1})$, in particular the curvature is $O(r^{-3})$. If we allow conical singularities up to infinity, then we must generalize the definition to admit transverse metrics $\gamma$ with conical singularities as well.

The volume growth is $r^3$, which corresponds to the fact that the direction of the vector field $T$ collapses at infinity. The definition is designed so that the orbits of $T$ do not need be closed. This does not occur in the hyperKähler case, but is the general case for purely Riemannian Ricci flat metrics (for example the Kerr metrics).

We now suppose that the metric is toric, that is there is an action of a 2-torus $\Bbb{T}^2$ on $(M^4,g)$ which preserves $g$. There is a compatibility condition between the ALF metric and the action:
\begin{defn}
  A Riemannian manifold $(M^4,g)$ is toric ALF if it is ALF, and the torus action near the end of $M$ is an action on $L$ which preserves $(\eta,T,\gamma)$, and such that $T$ is generated by the infinitesimal action of $\Bbb{T}^2$.
\end{defn}

We now introduce the Hermitian condition:
\begin{thm}[{\cite[Prop. 5]{Der83}}]\label{th:derdzinski}
  Let $(M^4,g)$ be an oriented Einstein manifold, whose selfdual Weyl tensor $W^+$, viewed as a trace-free symmetric endomorphism of $\Lambda^+M$, has a simple nonzero eigenvalue $\lambda$ and a double eigenvalue $-\frac \lambda2$.
  Denote $J$ the almost complex structure such that the selfdual 2-form $\omega=g(J\cdot,\cdot)$ is an eigenform of $W^+$ for the eigenvalue $\lambda$. Then $(g_K:=\lambda^{\frac23}g, J, \omega_K=\lambda^{\frac23}\omega)$ is Kähler. In particular $J$ is a complex structure on $(M,g)$.
\end{thm}
The condition on the Weyl tensor is known as half algebraically special, or half type D. We choose the orientation of $M$ so that $W^+$ (rather than $W^-$) is degenerate.

Another equivalent way to state the same condition on $W^+$ is to say that $(M^4,g)$ admits a twistor 2-form, that is a selfdual 2-form $\tau$ satisfying the equation (also known as the Killing-Yano equation)
\begin{equation}
 \mathcal{T}^g(\tau)_X:=\nabla_X\tau - \frac13 ( X \lrcorner d\tau - X \wedge \delta\tau ) = 0\label{eq:2}
\end{equation}
for all vector fields $X$, see \cite{Pon92} and Lemma 2 in \cite{ApoCalGau03}. 

One has $f^{-3}\mathcal{T}^{f^2g}(f^3\tau)=\mathcal{T}^g(\tau)$, and it follows that if $(M^4,g)$ is half algebraically special then $\tau=\lambda^{-\frac13}\omega=\lambda^{-1}\omega_K$ is a twistor 2-form for $g$ (and the converse is true).

\begin{exa}
  In $\Bbb{R}^3\times S^1$ we have the twistor 2-form $\tau=r(-dr\wedge dt+r^2 \vol_{S^2})$.
\end{exa}
The next proposition proves that this is the only possible asymptotic behaviour for non hyperKähler toric ALF instantons:
\begin{prop}\label{prop:twistor-2-forms}
  Let $(M^4,g)$ be a toric Ricci flat ALF manifold, and $\tau$ a twistor 2-form. Suppose $W^+$ is nonzero. Then up to a nonzero multiplicative constant, $\tau$ has the asymptotic behaviour
  \begin{equation}
 \tau = r(-dr\wedge \eta + r^2 \vol_\gamma) + O(1).\label{eq:3}
\end{equation}
\end{prop}
Observe that this formula implies that we have chosen the sign of $\eta$ so that an oriented coframe is obtained from the forms $(-dr,\eta)$ and from an oriented coframe of $\ker \eta$.
\begin{proof}
  We begin by solving the twistor equation for $\Bbb{T}^2$-invariant selfdual 2-forms on $\Bbb{R}^3\times S^1$. We take radial coordinates $(r,\theta,\phi)$ on $\Bbb{R}^3$ and an angular coordinate $t$ on $S^1$, so the metric of $\Bbb{R}^3\times S^1$ is
  \begin{equation}
    \label{eq:4}
    g_0 = dr^2 + r^2 (d\theta^2 + \sin^2\theta d\phi^2) + dt^2.
  \end{equation}
  The torus is given by the angular variables $(\phi,t)$. We introduce the orthonormal coframe $\theta_1=dr$, $\theta_2=r d\theta$, $\theta_3=r \sin \theta d\phi$, $\theta_4=dt$. These are invariant under the torus action. We also use the basis of selfdual forms (hidden here is a choice of orientation of $\Bbb{R}^3\times S^1$):
  \begin{equation}
    \label{eq:5}
       \omega_1 = - \theta_1 \wedge \theta_4 + \theta_2 \wedge \theta_3, \quad \omega_2 = \theta_1 \wedge \theta_2 - \theta_3 \wedge \theta_4, \quad \omega_3 = \theta_1 \wedge \theta_3 + \theta_2 \wedge \theta_4.
     \end{equation}
     We take an invariant selfdual form $\tau=\sum_1^3 a_i(r,\theta) \omega_i$. The reader can check that the twistor equation (\ref{eq:2}) is given by the system of the following 8 equations:
     \begin{equation}
       \label{eq:6}
       \begin{split}
         r(a_1)_r &= a_1 - \cot(\theta)a_3 = a_1 - (a_3)_\theta \\
         r(a_2)_r &= a_2 = 0 \\
         r(a_3)_r &= a_3+(a_1)_\theta = 0 \\
         (a_2)_\theta &= \cot(\theta)a_2 = 0
       \end{split}
     \end{equation}
     It is immediate that the solutions are $a_1=k_1 \cos(\theta)+k_2r$, $a_2=0$, $a_3=k_1\sin(\theta)$ for some constants $k_1$ and $k_2$. So we have two solutions: $\tau_1=r\omega_1$ and $\tau_2=\cos(\theta)\omega_1+\sin(\theta)\omega_3$. Actually $\tau_2=- d(r\cos\theta)\wedge \theta_4 + d(r\sin\theta)\wedge \theta_3$ is parallel, since $r\cos(\theta)$ is the coordinate on the invariant axis of $\Bbb{R}^3$ and $d(r\sin \theta)\wedge\theta_3$ is the volume form in the orthogonal plane. In particular it is bounded, and we have proved (\ref{eq:3}) for the flat model $\Bbb{R}^3\times S^1$.

     The case of a general ALF metric can be deduced: starting from the behaviour (\ref{eq:1}), we write locally $\eta=dt+\alpha$ where $dt(T)=1$ and $\alpha$ is a 1-form on the local quotient by the action of $T$. Then with respect to $g$ one has $\alpha=O(r^{-1})$. Taking also local spherical coordinates $(\theta,\phi)$ on this local quotient, such that $\partial_\theta$ is a second generator of $\Bbb{t}^2$ (the first being $T$), we can write locally the ALF metric as $g=g_0+O(r^{-1})$, where $g_0$ is our flat metric (\ref{eq:4}) on $\Bbb{R}^3\times S^1$. The twistor equation for $g$ is the same as for $g_0$ with $O(r^{-1})$ perturbation, but more importantly the twistor operator $\mathcal{T}$ has the same symbol: this implies that the system (\ref{eq:6}) includes a set of ODE's for $(a_1,a_2,a_3)$, that is $(\partial_r a_1,\partial_r a_2,\partial_r a_3)$ is completely determined in terms of $(a_1,a_2,a_3)$ (without derivative). It follows that a solution $\tau$ for $g$ is asymptotic to a solution for $g_0$ up to $O(r^{-1})$ terms. In particular, since $|\tau_1|=r$,  for some constant $c$ one has
     \begin{equation}
       \label{eq:7}
       \tau = c \tau_1 + O(1).
     \end{equation}
There remains to prove that $c\neq0$. If $c=0$, then we get instead $\tau = c' \tau_2 + O(r^{-1})$. As we have seen, up to a constant, we would then have that $\tau=\lambda^{-\frac13} \omega$ with $|\omega|^2=2$ an eigenform of $W^+$ for the eigenvalue $\lambda$. Since by (\ref{eq:7}) we have $|\tau|=O(1)$ we would obtain $\lambda^{-1}=O(1)$. By (\ref{eq:1}) we have $|W|=O(r^{-3})$ so this is impossible, except if $\lambda\equiv0$ that is if $W^+\equiv0$.
\end{proof}

\section{Toric Kähler geometry}
\label{sec:toric-kahl-geom}

We start from a toric Ricci flat ALF manifold $(M,g)$ which is Hermitian, that is satisfies the condition of theorem \ref{th:derdzinski}. We therefore have a complex structure $J$ and a conformal metric $g_K=\lambda^{\frac23}g$ which is Kähler. Moreover, from Proposition \ref{prop:twistor-2-forms} we deduce that for some constant $k\neq0$ we have $\lambda=\frac k{r^3}+O(\frac1{r^4})$. We change the conformal factor $\lambda^{\frac23}$ to obtain the following asymptotic behaviour for the Kähler metric and the Kähler form as $r\rightarrow+\infty$:
\begin{align}
  g_K &= (k^{-1}\lambda)^{\frac23}g = \frac{dr^2+\eta^2}{r^2} + \gamma + O_{g_K}(\frac1r), \label{eq:8}\\
 \omega_K &= (k^{-1}\lambda)\tau = - \frac{dr\wedge \eta}{r^2} + \vol_\gamma + O_{g_K}(\frac1r),\label{eq:9}
\end{align}
which means that the Kähler metric (and its complex structure) is locally asymptotic to the product of a Poincaré disk (curvature $-1$) with a round sphere (curvature $+1$).

One has $W^+_{g_K}=(k^{-1}\lambda)^{-\frac23}W^+_g$ (seen as endomorphisms of $\Lambda^+$) and therefore the simple eigenvalue of $W^+_{g_K}$ is $k^{\frac23}\lambda^{\frac13}$. For a Kähler metric, this is also $\frac{\Scal}6$ so we have
\begin{equation}
\Scal_{g_K}=6k^{\frac23}\lambda^{\frac13}.\label{eq:10}
\end{equation}
In particular, $\Scal_{g_K}\rightarrow0$ when $r\rightarrow\infty$, which is compatible with the previous asymptotics.

A Kähler metric is \emph{extremal} if the scalar curvature is the Hamiltonian potential of a Killing vector field.
\begin{lem}
  If the ALF space $(M^4,g)$ admits a twistor 2-form $\tau$ satisfying (\ref{eq:3}), then the following are equivalent:
  \begin{itemize}
  \item $(M^4,g)$ is Ricci flat;
  \item $\frac{g}{|\tau|^3}$ is an extremal, Bach flat, Kähler metric.
  \end{itemize}
\end{lem}
\begin{proof}
  This lemma is well known, but it will be useful to recall the proof.
  Since $\Ric^g=0$, after the conformal change we obtain
\begin{equation}
  \label{eq:11}
  \Ric^{g_K}_0 = -2 \frac{(\nabla^{g_K}d\lambda^{\frac13})_0}{\lambda^{\frac13}}.
\end{equation}
Since $g_K$ is Kähler, it follows that $\nabla^{g_K}d\lambda^{\frac13}$ is $J$-invariant and therefore the (real) vector field $\nabla^{g_K}\lambda^{\frac13}$ is holomorphic. In particular, the Hamiltonian vector field $T=J \nabla^{g_K}(k^{-1}\lambda)^{\frac13}$ is Killing with respect to $g_K$ and therefore with respect to $g$ as well. (This vector is the vector $T$ of section \ref{sec:alf-metrics}).

If $x_1$ is the moment of $T$ with respect to $\omega_K$ ($dx_1=-T \lrcorner \omega_K$), we have
\begin{equation}
  \label{eq:12}
  x_1=(k^{-1}\lambda)^{\frac13}, \qquad g_K=x_1^2 g, \qquad \Scal_{g_K}=6 k x_1.
\end{equation}
The scalar curvature is therefore the Hamiltonian potential of a Killing vector field, which means that $g_K$ is extremal. It is also Bach flat since it is conformal to a Ricci flat metric.

The converse is well-known, see for example \cite{CheLebWeb08}: since $g_K$ is extremal and Bach flat, one obtains that $g$ is Einstein, maybe with nonzero Einstein constant; the asymptotics (\ref{eq:1}) implies the decay $\Ric=O(r^{-3})$, so the Einstein constant has to vanish.
\end{proof}

At this point the machinery of toric Kähler geometry applies for $(M,g_K)$. The moment map $\mu:M\rightarrow (\Bbb{t}^2)^*=\Bbb{R}^2$ has image a convex polytope $P$. Every edge of the polytope is the image of a divisor fixed by an element of $\Bbb{t}^2$; the vertices are the fixed points of the action. If we take an integral basis $(\partial/\partial\theta_1,\partial/\partial\theta_2)$ of $\Bbb{t}^2$, then the corresponding moment maps $(\mu_1,\mu_2)$ are affine coordinates on $P$. The Kähler form takes the form (action-angle coordinates):
\begin{equation}
  \label{eq:13}
  \omega_K = d\mu_1 \wedge d\theta_1 + d\mu_2 \wedge d\theta_2 .
\end{equation}
There is a convex function $\psi$ on $P$, the symplectic potential, so that the metric takes the form
\begin{equation}
  \label{eq:14}
  g_K = \sum _{i, j = 1} ^ 2  (\psi_{\mu_i \mu_j} d\mu_id\mu_j + \psi^{\mu_i \mu_j} d\theta_id\theta_j),
\end{equation}
where $\psi_{\mu_i \mu_j}=\frac{\partial^2\psi}{\partial \mu_i \partial \mu_j}$ and $(\psi^{\mu_i \mu_j})=(\psi_{\mu_i \mu_j})^{-1}$. The scalar curvature is given by Abreu's formula,
\begin{equation}
  \label{eq:15}
  \Scal_{g_K} = - \sum _{i, j = 1} ^2  \psi^{\mu_i \mu_j}_{\mu_i \mu_j},
\end{equation}
and the metric is extremal if and only this is the Hamiltonian function of a Killing vector field, that is if and only if $\Scal_{g_K}$ is an affine function.

The complex structure is given in the interior of $P$ by the holomorphic coordinates $z_j=\exp(\nu_j+i \theta_j)$ with $\nu_j=\psi_{\mu_j}$. It is useful to note that we always have a global holomorphic form above the interior of $P$, which extends over the whole $M$ as a meromorphic form with simple poles along the divisors:
\begin{equation}
  \label{eq:16}
  \Omega = -4i \partial\theta_1 \wedge \partial\theta_2 = i(d\nu_1+id\theta_1)\wedge(d\nu_2+id\theta_2)
\end{equation}
(This normalization will be used later).

For $M$ to be smooth we need $P$ to be a Delzant polytope: each edge $E$ must be rational, and in particular has an integral equation $\lambda_E$. This means:
\begin{itemize}
\item $\lambda_E$ is an affine function on $(\Bbb{t}^2)^*$ vanishing on $E$; so we can write $\lambda_E(x)=\langle v_E,x\rangle+c_E$ for some $v_E\in\Bbb{t}^2$ and $c_E\in\Bbb{R}$;
\item the normal $v_E=d\lambda_E$ is a primitive element of $\Bbb{Z}^2 \subset \Bbb{t}^2$, and it is an inward normal (so that $\lambda_E>0$ in the interior of $P$).
\end{itemize}
Finally it is requested that at each vertex $E_1\cap E_2$, the two inward normals $v_{E_1}$, $v_{E_2}$ form an integral basis of $\Bbb{Z}^2 \subset \Bbb{t}^2$. In our case of course this does not apply to the edge at infinity $E_\infty$ which can be irrational (this corresponds to the vector field $T$ having non closed orbits).

Then the smoothness of the metric $g_K$ is read on its symplectic potential:
\begin{itemize}
\item near an edge $E$ one must have $\psi=\frac12 \lambda_E \log \lambda_E + $smooth function;
\item near a vertex $E_1\cap E_2$ one must have $\psi=\frac12 \sum_{j =1} ^2 \lambda_{E_j} \log \lambda_{E_j} + $smooth function.
\end{itemize}
A conical singularity of angle $2\pi\alpha_E$ around $E$ is obtained by replacing $\lambda_E$ by $\frac1{\alpha_E}\lambda_E$ in the above behaviour of $\psi$.

Finally, we give the idea to obtain a Poincaré behaviour (\ref{eq:8}) near the edge at infinity $E_\infty$, but here it seems that there is no nice regularity result in the literature so we will actually use a much weaker condition, see section \ref{sec:case-edge-at}. Choose an affine equation $x_1$ of $E_\infty$, and suppose:
\begin{itemize}
\item near $E_\infty$ one has $\psi=-\log x_1 + a x_1 \log x_1 + $smooth function for some constant $a$: this does not depend on the choice of $x_1$, which can be irrational;
\item near a vertex at infinity $E_\infty\cap E_j$ we need $\psi=-\log x_1 + a x_1 \log x_1+\frac12 \lambda_{E_2} \log \lambda_{E_2} + $smooth function.
\end{itemize}
Then the corresponding metric has Poincaré behaviour near $E_\infty$. The function $x_1$ here is the same as in (\ref{eq:12}): it is the Hamiltonian function of a Killing vector field, and vanishes on the boundary at infinity.

\section{The Tod ansatz}
\label{sec:ansatz}

In this section, we present a quick derivation of the ansatz of Tod \cite{Tod20}. In contrast with \cite{Tod20}, our approach initially relies on a well-known K\"ahler ansatz due to LeBrun, but like in \cite{Tod20} ends up with a celebrated B\"aklund transformation introduced by Ward \cite{War90}. 

\subsection{The first ansatz}
\label{sec:first-ansatz}

We start from a Kähler structure $(g_K,J,\omega_K)$ with a Hamiltonian Killing vector field $T$. We can then apply the LeBrun ansatz:
\begin{prop}[{\cite[§2]{LeB91b}}]\label{prop:lebrun_ansatz}
  One can write $g_K$ and $\omega_K$ under the form
  \begin{align}
    \label{eq:17}
    g_K &= W \big( dx_1^2 + e^v(dx_2^2+dx_3^2) \big) + W^{-1} \eta^2 \\
    \omega_K &= dx_1 \wedge \eta + We^v dx_2 \wedge dx_3,\label{eq:18}
  \end{align}
  where the 1-form $\eta$ and the function $W=W(x_1,x_2,x_3)=|T|_{g_K}^{-2}$ satisfy $T\lrcorner \eta=1$ and
  \begin{gather}
    \label{eq:19}
    d\eta = (We^v)_{x_1} dx_2\wedge dx_3 + W_{x_2} dx_3\wedge dx_1 + W_{x_3} dx_1\wedge dx_2,\\
    \label{eq:20}
    (We^v)_{x_1 x_1} + W_{x_2 x_2} + W_{x_3 x_3} = 0.
  \end{gather}
 The second equation is the compatibility condition $d(d\eta)=0$.  Finally,
  \begin{equation}
    \label{eq:21}
    \Scal_{g_K} = - \frac{ (e^v)_{x_1 x_1}+v_{x_2 x_2}+v_{x_3 x_3}}{We^v}.
  \end{equation}
\end{prop}
\begin{rem}\label{rem:def-v} The coordinates $x_2, x_3$ and the function $v$ are determined by the choice of a local, non-vanishing holomorphic  $T$-invariant 2-form $\Omega = e ^{- \frac{v}{2}}  (\omega _2 + i \omega _3)$, where $(\omega _K, \omega _2, \omega _3)$ form a direct  (local) orthonormal  frame of $\Lambda ^+M$ (with $|\omega_j|^2=2$) and $e ^{- v} = \frac{1}{4} |\Omega| ^2 _{g _K}$;  $\Omega$ being holomorphic is closed, and so is $T\lrcorner \Omega$, since $\Omega$ is $T$-invariant;,  we thus define  $x_2, x_3$  by $T \lrcorner \Omega = - (d x_2 + i d x_3)$, while $\omega _2, \omega _3$ are given by $\omega _2 = e ^{\frac{v}{2}} (d x_2 \wedge \eta + W \, d x_3 \wedge d x_1)$, $\omega _3 = e ^{\frac{v}{2}} (d x_3 \wedge \eta + W \, d x_1 \wedge d x _2)$; since $e ^{- \frac{1}{2}} \omega _2$, $e ^{- \frac{1}{2}} \omega _3$ are closed, as well as $\omega _K$,  (\ref{eq:19}), hence (\ref{eq:20}), readily follow. 
   In our toric setting, when $\eta = dt - F \, d x _3$ and $v, W, F$ are functions of $x_1, x_2$ only, see section \ref{sec:second-ansatz} below, the formula (\ref{eq:16}) gives $x_2=\nu_2$ and $x_3=\theta_2$, $d \nu _1 = W \, d x_1 + F \, d x _2$ and $\theta _1 = t$. 
\end{rem}
We deduce
\begin{prop}
  Let $(g_K,J,\omega_K)$ be a Kähler structure as in proposition \ref{prop:lebrun_ansatz}. Then the conformal metric $g=x_1^{-2}g_K$ is Ricci flat if and only if
  \begin{equation}
 W=\frac2{kx_1^3}-\frac{v_{x_1}}{kx_1^2}\label{eq:22}
\end{equation}
and $v$ satisfies the twisted Toda equation
\begin{equation}\label{eq:23}
  (e^v)_{x_1 x_1}+v_{x_2 x_2}+v_{x_3 x_3} = -6k x_1 W e^v.
\end{equation}
\end{prop}
\begin{proof}
  Since $g=x_1^{-2}g_K$ we have $\Scal_g=x_1^2\Scal_{g_K}-6x_1\Delta_K x_1-12|dx_1|^2_K$ and therefore
  \begin{equation}
    \label{eq:24}
    k x_1^3 - 6 x_1 \Delta_K x_1 - 12 |dx_1|_K^2 = 0.
  \end{equation}
  We calculate $\Delta_Kx_1$: one has $d^Cx_1=\frac \eta W$ so
  \begin{equation}
    \label{eq:25}
    dd^Cx_1 = \frac{d\eta}W - \frac{dW\wedge \eta}{W^2}.
  \end{equation}
  Using (\ref{eq:19}) we thus get $\Delta_Kx_1 = - \frac{v_{x_1}}W$. Injecting this value in (\ref{eq:24}) we obtain (\ref{eq:22}). Equation (\ref{eq:23}) is just (\ref{eq:20}) with $\Scal_{g_K}=6kx_1$.

  Conversely, if we have (\ref{eq:22}) and (\ref{eq:23}) it follows that $\Scal_{g_K}=6kx_1$ and $\Scal_g=0$. Since $\Ric_0^g=\Ric_0^{g_K}+2x_1^{-1}(\nabla_Kdx_1)_0$, we infer that $g$ is Ricci flat if $\rho_0^{g_K}+x_1^{-1}(dd^Cx_1)_0=0$, where the Ricci form $\rho^{g_K}$ is the (1,1)-form $\Ric^{g_K}(J\cdot,\cdot)$. By remark \ref{rem:def-v} we have a holomorphic volume form $\Omega$ with $|\Omega|^2_K=4e^{-v}$ so $\rho^{g_K}=-\frac12 dd^Cv$, so finally $g$ is Ricci flat if
  \begin{equation}
    \label{eq:26}
    \frac12 (dd^Cv)_0 + \frac{(dd^Cx_1)_0}{x_1} = 0.
  \end{equation}
  The calculation of (\ref{eq:26}) given (\ref{eq:22}) and (\ref{eq:23}) is left to the reader.
\end{proof}

One can transform (\ref{eq:23}) into the usual Toda equation by setting:
\begin{equation}
  \label{eq:27}
  V=x_1^2W, \quad \xi=\frac1{x_1}, \quad u=v-4 \log x_1.
\end{equation}
We then get the following formulas:
\begin{gather}
  V=\frac1k (-2\xi + \xi^2u_\xi), \label{eq:28}\\
  (e^u)_{\xi\xi}+u_{x_2 x_2}+u_{x_3 x_3}=0, \label{eq:29}\\
  g = \frac{g_K}{x_1^2} = V\big( d\xi^2 + e^u(dx_2^2+dx_3^2) \big) + V^{-1}\eta^2, \label{eq:30}\\
  d\eta = -\xi^2 (\xi^{-2}e^uV)_\xi dx_2\wedge dx_3 - V_{x_2} dx_3\wedge d\xi - V_{x_3} d\xi\wedge dx_2.\label{eq:31}
\end{gather}

\subsection{The second ansatz}
\label{sec:second-ansatz}
We now suppose that the previous data does not depend on $x_3$. It follows that the 1-form $\eta$ can be put in the form
\begin{equation}
  \label{eq:32}
  \eta=dt - F dx_3
\end{equation}
for some choice of the function $t$, which completes the coordinate system $\xi$, $x_2$, $x_3$, and $F=F(\xi,x_2)$. The Hamiltonian vector field is $T$ is simply $\frac \partial{\partial t}$, and we have another commuting Hamiltonian Killing vector field, $\frac \partial{\partial x_3}$, so we are in the toric situation. The Toda equation (\ref{eq:29}) becomes
\begin{equation}
  \label{eq:33}
  (e^u)_{\xi\xi} + u_{x_2 x_2} = 0.
\end{equation}
From (\ref{eq:31}) (\ref{eq:28}) and (\ref{eq:33}) it follows that
\begin{equation}
  \label{eq:34}
  F_\xi=-V_{x_2}=-\frac1k \xi^2u_{\xi x_2}, \qquad F_{x_2}=\xi^2(\xi^{-2}e^uV)_\xi=-\frac1k\big( \xi^2u_{x_2 x_2}-2e^u+2\xi e^uu_\xi \big).
\end{equation}
The compatibility equation is simply the Toda equation (\ref{eq:33}).

We now perform the Ward ansatz. Take standard coordinates $(u_1,u_2,u_3)$ on $\Bbb{R}^3$ and define $\rho=\sqrt{u_1^2+u_2^2}$ and $z=u_3$. The main statement is the following:
\begin{prop}
  Suppose that $U$ is a harmonic axisymmetric function on $\Bbb{R}^3$, that is satisfies the equation
  \begin{equation}
    \label{eq:35}
    \frac1\rho(\rho U_\rho)_\rho + U_{zz} = 0.
  \end{equation}
  If we define
  \begin{equation}
    \label{eq:36}
    \xi = \frac12 \rho U_\rho, \qquad x_2 = - \frac12 U_z,
  \end{equation}
  then the function $u=\log \rho^2$ is a solution to the reduced Toda equation (\ref{eq:33}).
\end{prop}
\begin{proof}
  This is proved by calculation. We have the first derivatives
  \begin{equation}
 \xi_\rho=-\frac12 \rho U_{zz}, \quad \xi_z=\frac12 \rho U_{\rho z}, \quad(x_2)_\rho=-\frac12 U_{\rho z}, \quad (x_2)_z=-\frac12 U_{zz}.\label{eq:37}
 \end{equation}
We invert this into
  \begin{equation}
    \label{eq:38}
    \rho_\xi = -\frac{U_{zz}}{2\Delta}, \quad \rho_{x_2} = - \frac{\rho U_{\rho z}}{2 \Delta}, \quad
    z_\xi = \frac{U_{\rho z}}{2\Delta}, \quad z_{x_2} = - \frac{\rho U_{zz}}{2\Delta},
  \end{equation}
  where $\Delta=\frac \rho4(U_{zz}^2+U_{\rho z}^2)$. A straightforward calculation using (\ref{eq:38}) leads to
  \begin{equation}
    \label{eq:39}
    (e^u)_{\xi\xi}+u_{x_2x_2} = -\frac{2\Delta_\rho}{\Delta^2}+\frac{U_{zz}^2}{2\Delta^2}
    + \frac{\rho U_{\rho zz}U_{zz}}{\Delta^2} + \frac{\rho U_{\rho\rho z}U_{\rho z}}{\Delta^2}
    - \frac{\rho U_{zzz}U_{\rho z}}{2\Delta^2}.
  \end{equation}
  But $\Delta_\rho=\frac14(U_{\rho z}^2+U_{zz}^2)+\frac \rho2(U_{zz}U_{\rho zz}+U_{\rho z}U_{\rho\rho z})$, so we get
  \begin{equation}
    \label{eq:40}
    \begin{split}
      (e^u)_{\xi\xi}+u_{x_2x_2} &= - \frac{U_{\rho z}}{\Delta^2}(U_{\rho z}+\rho U_{\rho\rho z}+\rho U_{zzz}) \\
      &= - \frac{U_{\rho z}}{2\Delta^2}(\rho U_{\rho\rho}+\rho U_{zz}+U_\rho)_z = 0.
    \end{split}
  \end{equation}
\end{proof}

\begin{cor}[Tod ansatz]\label{cor:harmarck}
  Under the same notations as in the proposition, from an axisymmetric harmonic function $U$ we deduce a Hermitian toric Ricci flat metric given by the formulas
  \begin{equation}
    \label{eq:41}
    g=e^{2\nu}(d\rho^2+dz^2) + V\rho^2 dx_3^2 + V^{-1} (dt-F dx_3)^2,
  \end{equation}
  where the functions $\nu$, $V$ and $F$ are given by:
  \begin{align}
    e^{2\nu} &= \frac14 V \rho^2 (U_{\rho z}^2+U_{zz}^2), \label{eq:42}\\
    V  &= - \frac1k \left( \rho U_\rho + \frac{U_\rho^2U_{zz}}{U_{\rho z}^2+U_{zz}^2} \right), \label{eq:43}\\
    F &= - \frac1k \left( -\frac{\rho U_\rho^2 U_{\rho z}}{U_{\rho z}^2+U_{zz}^2} + \rho^2 U_z +2H \right).\label{eq:44}
  \end{align}
  The function $H=H(\rho,z)$ is a conjugate of the harmonic function $U$: it is determined up to an additive constant by
  \begin{equation}
    \label{eq:45}
    H_\rho = - \rho U _z, \qquad H_z = \rho U_\rho.
  \end{equation}
  The formulas make sense on the domain where $V>0$ and $U_{\rho z}^2+U_{zz}^2>0$.
\end{cor}

We will call $U$ the \emph{generating function} of the metric $g$.
\begin{proof}
  It is a matter of putting everything together. From (\ref{eq:37}) we infer
  \begin{equation}
    \label{eq:46}
    d\xi^2 + e^u dx_2^2 = \frac{\rho^2}4 (U_{zz}^2+U_{\rho z}^2)(d\rho^2+dz^2).
  \end{equation}
  Then (\ref{eq:41}) follows from (\ref{eq:30}), with $e^{2\nu}$ given by (\ref{eq:42}); and (\ref{eq:43}) follows from (\ref{eq:28}) and (\ref{eq:38}). There remains to obtain the expression of $F$ in terms of $\rho$, $z$, and here we follow a trick of Tod: from (\ref{eq:34}) it is natural to introduce a function $H$ by
  \begin{equation}
    \label{eq:47}
    F = - \frac1k ( \xi^2 u_{x_2} - 2x_2 e^u + 2H),
  \end{equation}
  and $H$ satisfies
  \begin{equation}
    \label{eq:48}
    \begin{split}
      H_\xi &= -\xi u_{x_2}+x_2(e^u)_\xi = \frac \rho{2\Delta} (U_\rho U_{\rho z}+\rho U_z U_{zz}), \\
      H_{x_2} &= \xi(e^u)_\xi + x_2 (e^u)_{x_2} = - \frac{\rho^2}{2\Delta} (U_\rho U_{zz} + U_z U_{\rho z}).
    \end{split}
  \end{equation}
  From (\ref{eq:37}) again we then obtain (\ref{eq:45}) and (\ref{eq:44}).
\end{proof}
\begin{rem}
  In the toric case, (\ref{eq:41}) is the Harmark form of the metric which is well known in the physics literature, see for example \cite{Har04}. In particular $(\rho,z)$ are coordinates which give the uniformization of the metric $\sum _{i, j = 1} ^ 2 \psi_{\lambda_i \lambda_j}d\lambda_id\lambda_j$ on $P$. The corollary is a priori a local statement, but we will see in section \ref{sec:uniformization} that $(\rho,z)$ are global coordinates on the upper half-plane which uniformize $P$.
\end{rem}

\begin{rem}\label{rem:par-k}
  The parameter $k$ is a scale parameter, the metric obtained for the value $k$ is the same as the metric obtained for the value $1$ with a scale factor $\frac1k$. It is nevertheless useful to keep it since it influences the scale of the Kähler metric $g_K=x_1^2 g$. Since we want $g_K$ to be locally asymptotic to a product of a hyperbolic cusp with a round sphere, we will need to fix $k$ accordingly.
\end{rem}

\begin{rem}\label{rem:homogeneity}
  There is a homogeneity in the formulas (\ref{eq:41})--(\ref{eq:45}): if we consider the transformation $h_a(\rho,z)=(|a|\rho,az+b)$ (where $a$ can be negative), then the harmonic function $|a|^{-1} h_a^*U$ leads to a metric $|a|^{-3} H_a^*g$, where $H_a(\rho,z,t,x_3)=(|a|\rho,az+b,a^2t,\sign(a)x_3)$.
\end{rem}

\section{Boundary conditions}
\label{sec:boundary-conditions}

We now come back to our problem, with a toric extremal Kähler metric $(g_K,J,\omega_K)$, with polytope $P$, and conformal to a Ricci flat metric $g$. We analyze the boundary conditions on $\partial P$ and deduce the boundary conditions for the harmonic function $U$.

In all this section, we determine equivalents of the various objects near the boundary, and we also take derivatives of these equivalents. While a priori this is not correct, all the statements can be justified because we have a precise development of the symplectic potential $\psi$ near the boundary, which allows taking derivatives.

\subsection{The case of a finite edge}
\label{sec:case-finite-edge}

We begin with the case of an edge $E$ of $P$ with integral equation $\lambda_1=\lambda_E$. Near an interior point of $E$ we have $\psi=\frac12 \lambda_1 \log \lambda_1 + $smooth function. Completing $\lambda_1$ into a basis $(\lambda_1,\lambda_2)$ we obtain the behaviour of the metric on $P$ near $E$:
\begin{equation}
 \sum _{i, j = 1} ^2 \psi_{\lambda_i \lambda_j} d\lambda_i d\lambda_j \sim \frac{d\lambda_1^2}{2\lambda_1}  + \psi_{\lambda_2 \lambda_2} d\lambda_2^2,\label{eq:49}
\end{equation}
and it follows from (\ref{eq:16}) and (\ref{eq:27}) that
\begin{equation}
e^{-v} = \frac14 |\Omega|^2 \sim \frac{\psi_{\lambda_2 \lambda_2}}{2\lambda_1}, \quad \rho=e^{\frac u2}=\frac{e^{\frac v2}}{x_1^2}\sim \frac1{x_1^2}\sqrt{\frac{2\lambda_1}{\psi_{\lambda_2 \lambda_2}}}. \label{eq:50}
\end{equation}
Choosing the standard orientation of $P$ such that if $E$ and $E'$ are successive edges in the trigonometric order, then $(\lambda_E,\lambda_{E'})$ is oriented, and assuming that $\lambda_2$ was chosen so that $(\lambda_1,\lambda_2)$ is oriented, we deduce
\begin{equation}\label{eq:51}
  \begin{split}
    d\rho &\sim \frac{d\lambda_1}{x_1^2\sqrt{2\lambda_1\psi_{\lambda_2 \lambda_2}}} + \sqrt{2\lambda_1} \big( \frac1{x_1^2\sqrt{\psi_{\lambda_2 \lambda_2}}} \big)_{\lambda_2} d\lambda_2 , \\
    dz = - *_2 d\rho &\sim -\frac{d\lambda_2}{x_1^2} + \frac1{\sqrt{\psi_{\lambda_2 \lambda_2}}} \big( \frac1{x_1^2\sqrt{\psi_{\lambda_2 \lambda_2}}} \big)_{\lambda_2} d\lambda_1.
  \end{split}
\end{equation}
The Hodge star in the second equation is the Hodge operator for the metric $\sum \psi_{\lambda_i \lambda_j}d\lambda_id\lambda_j$ induced on $P$. It follows from this formula that the function $z$ attains finite values on $E$. Since along the boundary $\partial P$ we have
\begin{equation}
\frac{\partial z}{\partial \lambda_2}=-x_1^{-2}<0,\label{eq:52}
\end{equation}
we deduce that $z$ is monotone on $\partial P\setminus E_\infty$, where $E_\infty$ is the edge at infinity ($x_1=0$). Actually if we go along $\partial P\setminus E_\infty$ in the trigonometric sense, the function $z$ is increasing; moreover when we reach $E_\infty$ we see that $z\rightarrow \pm \infty$, so $z$ is actually increasing from $-\infty$ to $+\infty$ on $\partial P\setminus E_\infty$. 

The equation (\ref{eq:50}) says that $\rho|_E=0$. From (\ref{eq:36}) we have $\rho U_\rho = \frac 2{x_1}$ and therefore near $E$ we have
\begin{equation}
  \label{eq:53}
  U \sim \frac1{x_1|_E} \log \rho^2. 
\end{equation}
From (\ref{eq:52}) it follows that $\frac1{x_1|_E}$ is actually an affine function of $z|_E$. It follows that there is a continuous piecewise linear function $f:\Bbb{R}\rightarrow \Bbb{R}_+^*$, which is the same as $\frac1{x_1}$ if we see $z$ as a parameter of $\partial P$. Near each edge one has
\begin{equation}
  \label{eq:54}
  U \sim f(z) \log \rho^2 .
\end{equation}
We will prove later the crucial observation that the slopes actually increase from $-1$ to $+1$. In particular the function $f$ is convex.



\subsection{The case of a finite vertex}
\label{sec:case-finite-vertex}

This is similar to the previous case: if we have a vertex $E\cap E'$, where $(E,E')$ are in trigonometric order, then we can take $\lambda_1=\lambda_E$ as above and $\lambda_2=\lambda_{E'}$. Then near the vertex one has $\psi=\frac12 \sum_{i = 1} ^2 \lambda_i \log \lambda_i + $smooth function. We have the same formulas as above except that now $\psi_{\lambda_2 \lambda_2}=\frac1{2\lambda_2}+$smooth function.  Then (\ref{eq:50}) and (\ref{eq:51}) give
\begin{equation}
  \label{eq:55}
  \rho \sim \frac{2\sqrt{\lambda_1\lambda_2}}{x_1^2}, \quad
  dz \sim \frac{d\lambda_1-d\lambda_2}{x_1^2}.
\end{equation}
In particular the behaviour of $z$ and $U$ near the vertex remains the same as described above. Together with the previous case, it implies that the asymptotic of $U$ in (\ref{eq:54}) is global near $\{\rho=0\}=\partial P\setminus E_\infty$.

\subsection{The case of the edge at infinity}
\label{sec:case-edge-at}

Here we choose coordinates $(\lambda_1=x_1,\lambda_2)$, where $\lambda_2$ can be chosen to be $\lambda_2=\lambda_E$ for the edge $E$ going to infinity such that $(x_1,\lambda_2)$ is a direct basis. We actually choose $x_1$ normalized so that $(x_1,\lambda_2)$ is a basis of determinant $1$. We claim that the symplectic potential $\psi$ satisfies
\begin{equation}
  \label{eq:121}
  \psi(x_1,\lambda_2) = - \log x_1 + \varpi(\lambda_2) + \zeta
\end{equation}
where $\zeta$ is a function which vanishes on $x_1=0$, such that the second derivatives $f=x_1^2\zeta_{x_1x_1}, x_1\zeta_{x_1\lambda_2}$ or $\zeta_{\lambda_2\lambda_2}$ all satisfy
\begin{equation}
  \label{eq:122}
  (x_1\partial_{x_1})^{k_1}\partial_{\lambda_2}^{k_2}f = O(x_1) \text{ for all } k_1, k_2\geq0.
\end{equation}
We will justify these conditions below, but we first use them to establish the behaviour of $U$ at infinity.

It follows from (\ref{eq:121})--(\ref{eq:122}) that (defining $\varpi^{\lambda_2 \lambda_2}=\varpi_{\lambda_2 \lambda_2}^{-1}$):
\begin{equation}
  \label{eq:56}
  \sum _{i, j = 1} ^2  \psi_{\lambda_i \lambda_j}d\lambda_i d\lambda_j \sim \frac{dx_1^2}{x_1^2} + \varpi_{\lambda_2 \lambda_2} d\lambda_2^2,
  \qquad \rho \sim \frac{\sqrt{\varpi^{\lambda_2 \lambda_2}}}{x_1}.
\end{equation}
At this point, we note that because of equation (\ref{eq:12}) and Abreu's formula (\ref{eq:15}), we must have $\sum _{i, j = 1}^2 \psi^{\lambda_i \lambda_j}_{\lambda_i \lambda_j}=0$ along $x_1=0$, which translates into $\varpi^{\lambda_2 \lambda_2}_{\lambda_2 \lambda_2}=-2$. We have two edges $E$ and $E'$ going to infinity, and it follows from the description of the behavior at $E\cap E_\infty$ and $E'\cap E_\infty$ that we must have $\varpi^{\lambda_2 \lambda_2}=-\lambda_E^2+2 \lambda_E=-\lambda_{E'}^2+2 \lambda_{E'}$ along $E_\infty$, which implies that $\lambda_{E'}=2-\lambda_E$ along $E_\infty$ and that $E_\infty$ has length $2$ for $\lambda_E$ (or $\lambda_{E'}$). At the end we deduce that, up to some global affine function, we have
\begin{equation}
  \label{eq:57}
  \varpi(\lambda_2) = \frac12 \big( \lambda_2 \log \lambda_2 + (2-\lambda_2) \log (2-\lambda_2) \big).
\end{equation}

Coming back to the function $\rho$, we have
\begin{equation}
  \label{eq:58}
  d\rho \sim -\sqrt{\varpi^{\lambda_2 \lambda_2}} \frac{dx_1}{x_1^2}+ \frac{\varpi^{\lambda_2 \lambda_2}_{\lambda_2}}{2\sqrt{\varpi^{\lambda_2 \lambda_2}}} \frac{d\lambda_2}{x_1}, \quad
  dz \sim \frac{d\lambda_2}{x_1}+ \frac{\varpi^{\lambda_2 \lambda_2}_{\lambda_2}}2 \frac{dx_1}{x_1^2}.
\end{equation}
Observe that our equation $\varpi^{\lambda_2 \lambda_2}_{\lambda_2 \lambda_2}=-2$ is indeed the compatibility condition for the asymptotic terms of $dz$. Since $\varpi^{\lambda_2 \lambda_2}=\lambda_2(2-\lambda_2)$ on $E_\infty$ and $\varpi^{\lambda_2 \lambda_2}_{\lambda_2}=2(1-\lambda_2)$, we get  from (\ref{eq:56}) and (\ref{eq:58}) the asymptotics when $x_1\rightarrow 0$:
\begin{equation}
  \label{eq:59}
  \rho \sim \frac{\sqrt{\lambda_2(2-\lambda_2)}}{x_1}, \qquad z \sim \frac{\lambda_2-1}{x_1}.
\end{equation}
We obtain $R:=\sqrt{\rho^2+z^2}\sim x_1^{-1}$ and $\lambda_2\sim\frac{R+z}{R}$, and therefore from (\ref{eq:36})
\begin{equation}
  \label{eq:60}
  \begin{split}
    \rho U_\rho &= \frac{2}{x_1} \sim 2R \\
    U_z &= -2 \psi_{\lambda_2} \sim -2 \varpi_{\lambda_2} \sim \log \frac{2-\lambda_2}{\lambda_2} \sim \log \frac{R-z}{R+z}. 
  \end{split}
\end{equation}
Therefore we find the asymptotics of $U$ at infinity, that is when $R\rightarrow +\infty$ (which is the same as $x_1\rightarrow 0$):
\begin{equation}
  \label{eq:61}
  U \sim U_0(\rho,z):= 2R + z \log \frac{R-z}{R+z}. 
\end{equation}

Note that on the edge $E$ we have $\lambda_2=0$ and $z\rightarrow -\infty$, so we deduce $U \sim - z \log \rho^2$ when $\rho\rightarrow0$ and $R\rightarrow \infty$; since we have (\ref{eq:54}) when $\rho\rightarrow0$ with $f$ piecewise linear, we deduce that the slope of $f$ on the segment going to $-\infty$ is $-1$. Similarly, on the other edge $E'$ going to infinity we have $z\rightarrow +\infty$ and $U \sim z \log \rho^2$ when $\rho\rightarrow0$ and $R\rightarrow \infty$, therefore the slope of $f$ on the segment going to $+\infty$ is $+1$.


 We now justify the behaviour (\ref{eq:121})--(\ref{eq:122}). Since it seems that there is no regularity theory in the literature for extremal Kähler metrics with Poincaré behaviour, we instead obtain it directly from the definition (\ref{eq:1}) of an ALF metric. We have $g_K=x_1^2 g$ with $x_1=r^{-1}(1+O(r^{-1}))$ and the corresponding bounds on the derivatives. Therefore
  \begin{equation}
   g_K = \frac{dx_1^2}{x_1^2} + x_1^2 \eta^2 + \gamma + h, \quad |\nabla^k h|_{g_K} = O(x_1).\label{eq:62}
 \end{equation}
  We deduce the bounds on the Hessian of $\psi$:
  \begin{equation}
   \sum_{i,j=1}^2 \psi_{\lambda_i\lambda_j}d\lambda_id\lambda_j = \frac{dx_1^2}{x_1^2}+\varpi_{\lambda_2\lambda_2}d\lambda_2^2 + \sum_{ij=1}^2 A_{i,j} d\lambda_i d\lambda_j,\label{eq:63}
 \end{equation}
 where $\varpi(\lambda_2)$ is the symplectic potential of a spherical metric (that is $\varpi^{\lambda_2\lambda_2}_{\lambda_2\lambda_2}=-2$, an equation that we have already seen from another viewpoint), and $x_1^2A_{11}$, $x_1A_{12}$ and $A_{22}$ all satisfy $(x_1\partial_{x_1})^{k_1}\partial_{\lambda_2}^{k_2} f=O(x_1)$ for all $k_1,k_2\geq0$.

From the Hessian we deduce the symplectic potential itself: $\psi=-\log x_1+\varpi(\lambda_2)+\zeta$ with $x_1^2\zeta_{x_1x_1}, x_1\zeta_{x_1\lambda_2}, \zeta_{\lambda_2\lambda_2}=O(x_1)$; and the same bound for all derivatives of type $(x_1\partial_{x_1})^{k_1} \partial_{\lambda_2}^{k_2}$, which proves the claim. It is plausible that one can obtain a better regularity from the extremal equation, probably an asymptotic development in the variable $x_1$; in our case (with the additional Bach equation) no regularity theory is necessary since at the end the metrics will be explicit.



\subsection{The uniformization}
\label{sec:uniformization}

From their definition, the functions $(\rho,z)$ uniformize $\sum \psi_{\lambda_i \lambda_j}d\lambda_id\lambda_j$. We have the following global statement:
\begin{prop}\label{prop:unif}
  The functions $(\rho,z)$ give a diffeomorphism of the interior of $P$ with the upper half plane.
\end{prop}
\begin{proof}
  The metric $\gamma$ induced on the polytope $P$ can be uniformized by the upper half-plane $\Bbb{H}$. From the very definition of $z$ by (\ref{eq:51}), the function $\varphi=z+i\rho$ is a holomorphic function on $\Bbb{H}$, with real values on the real axis. By Schwarz reflection principle, the function $\varphi$ extends as a holomorphic function on $\Bbb{C}$. The point is to prove that it is an affine function.

  The behaviour of $\gamma$ at infinity (that is near $E_\infty$) is given by (\ref{eq:56}), with $\varpi(\lambda_2)=\frac1{\lambda_2(2-\lambda_2)}$. Taking $\sin \theta=\lambda_2-1$ and $\xi=x_1^{-1}$, we therefore have
  \begin{equation}
   \sum _{i, j = 1} ^2 \psi_{ij}d\lambda_id\lambda_j
    \sim   \frac{d\xi^2+ \xi^2 d\theta^2}{\xi^2}.\label{eq:64}
  \end{equation}
This is conformal to the hyperbolic metric $ \frac{d\xi^2+ \xi^2 d\theta^2}{\xi^2\cos^2\theta}$, with $\xi$ being the radius in the upper half plane. Since $|\varphi|=R \sim \xi$, the function $\varphi$ has linear growth at infinity, and is therefore affine.
\end{proof}

\subsection{The convexity of $f$}
\label{sec:convexity-f}

\begin{lem}\label{lem:sign-U}
  Suppose that $U$ is a solution of equation (\ref{eq:35}) on the upper half-plane $\rho>0$.
  \begin{itemize}
  \item If $U_\rho>0$ when $\rho\rightarrow 0$ and when $R=\sqrt{\rho^2+z^2}\rightarrow \infty$, then $U_\rho>0$ everywhere.
  \item If $U_{zz}<0$ when $\rho\rightarrow 0$ and when $R=\sqrt{\rho^2+z^2}\rightarrow \infty$, then $U_{zz}<0$ everywhere.
  \end{itemize}
\end{lem}
\begin{proof}
  Since $U_{zz}$ satisfies the same equation (\ref{eq:35}), the second statement is an immediate application of the maximum principle. For the first statement observe that $U_\rho$ satisfies $(U_\rho)_{\rho\rho} + \frac1\rho(U_\rho)_\rho + (U_\rho)_{zz}-\frac1{\rho^2}U_\rho=0$, and the maximum principle can also be applied to this equation.
\end{proof}

Now come back to our setting of a toric extremal Kähler metric $g_K$, with polytope $P$, and conformal to a Ricci flat ALF metric $g$. From proposition \ref{prop:unif} we have a global harmonic axisymmetric function $U$ on $\Bbb{R}^3$. The asymptotic (\ref{eq:54}) near $\partial P\setminus E_\infty$ implies that $U_\rho>0$ when $\rho\rightarrow0$. For the function $V$ in (\ref{eq:43}) to be positive we therefore must have $U_{zz} < 0$ when $\rho\rightarrow0$. On the other hand, when $R=\sqrt{\rho^2+z^2}\rightarrow\infty$, the asymptotic function $U_0(\rho,z)$ given (\ref{eq:61}) also has $U_\rho>0$ and $U_{zz}<0$. Therefore we deduce from lemma \ref{lem:sign-U} that $U_\rho>0$ and $U_{zz}<0$ everywhere. In particular for each $\rho$ the function $z\mapsto U(\rho,z)$ is concave, and at the limit $\rho\rightarrow0$ the limit $f(z)=\lim_{\rho\rightarrow0}\frac{U(\rho,z)}{\log \rho^2}$ is convex. Together with the calculations of the two limit slopes in section \ref{sec:case-edge-at} this gives:
\begin{cor}\label{cor:f-convex}
  For a toric extremal Kähler metric with polytope $P$, conformal to an ALF Ricci flat metric $g$, the piecewise linear function $f(z)$ given by (\ref{eq:54}) is convex, and has its slopes increasing from $-1$ to $+1$.\qed
\end{cor}
\begin{rem}\label{rem:angles}
  It is useful to note that the whole section extends to the case where the metric has conical singularities around the fixed point sets. This is because we then have exactly the same asymptotic behaviours with the equations $\lambda_E$ replaced by $\frac{\lambda_E}{\alpha_E}$, as explained in section \ref{sec:toric-kahl-geom}. 
\end{rem}

\section{The solutions}
\label{sec:the-solutions}

\subsection{The Taub-NUT generating function}
\label{sec:kerr-taub-bolt}

A significant role will be played by the following family of harmonic axisymmetric function parametrized by $n \geq 0$, namely, with $R := \sqrt{\rho ^2 + z ^2}$:
\begin{equation}
  \label{eq:65}
  U_n(\rho,z) = 2R - 2z \log \frac{R+z}{\rho} + 2n \log \rho^2.
\end{equation}
Notice that, when $\rho\rightarrow0$,  we have the behaviour (\ref{eq:54}) for the piecewise linear function
\begin{equation}
  \label{eq:66}
  f_n(z) = 2n + |z|.
\end{equation}
For $n > 0$, $U_n$ is actually a generating function of the (self-dual) Taub-NUT metric on parameter $n$. Using corollary \ref{cor:harmarck} we can reconstruct the metric. We normalize $k=4n$: this will be understood in section \ref{sec:regularity-criterion}, see (\ref{eq:86}). One calculates $H(\rho,z)=z(R+4n)+\rho^2 \log \frac{R+z}{\rho}$, then $V=e^{2\nu}=1+\frac{2n}R$ and $F = 2n \frac zR$. The formula (\ref{eq:41}) then gives
\begin{equation}
  \label{eq:67}
  g_n = \big( 1+\frac{2n}R \big) \big( d\rho^2+dz^2+\rho^2 dx_3^2 \big)
       + \frac1{1+\frac{2n}R} \big(dt - 2n \frac zR dx_3 \big)^2 .
\end{equation}
This is the Taub-NUT metric where the usual coordinates $(R,\theta)$ are related to $(\rho,z)$ by $z=R\cos \theta$ and $\rho=R\sin \theta$. When the parameter $n$ varies all these metrics are homothetic, so in this family there is only one metric up to scale. The function  $U_0 (\rho,z) = 2R (\rho, z) - 2z \log \frac{R (\rho, z) +z}{\rho}$ was already obtained in (\ref{eq:61}) as giving the asymptotics of any $U$: it is no longer a generating function but will play a prominent part in the expression of the general solution in the next subsection. Notice that its  conjugate function is $H_0 (\rho, z) = z R (\rho, z) + \rho ^2 \log{(R (\rho, z) + z)} - \rho ^2 \log{\rho}$, up to additive constant.
\subsection{Superposition}
\label{sec:superposition}

Suppose that $U$ is an axisymmetric harmonic function which satisfies $U_\rho>0$ and $U_{zz}<0$ everywhere (by lemma \ref{lem:sign-U} it is sufficient to have this sign at the boundaries). In particular the condition $U_{\rho z}^2+U_{zz}^2>0$ of corollary \ref{cor:harmarck} is satisfied. The condition $V>0$ in the same corollary is more difficult to achieve, but the following proposition shows that it is preserved when we superpose solutions: 

\begin{prop}\label{prop:superposition}
  Suppose that we have two solutions $U_1$ and $U_2$ of (\ref{eq:35}) which both satisfy $U_{zz}<0$, $U_\rho>0$ and
  \begin{equation}
   \rho U_\rho + \frac{U_\rho^2U_{zz}}{U_{\rho z}^2+U_{zz}^2} < 0.\label{eq:68}
 \end{equation}
  Then the same inequality remains true for any barycenter $a_1U_1+a_2U_2$ ($a_1, a_2>0$, $a_1+a_2=1$).
\end{prop}
\begin{proof}
  The inequality to be proved for $U=a_1U_1+a_2U_2$ is
  \begin{equation}
   U_{\rho z}^2 < - U_{zz} \big( \frac{U_\rho}\rho + U_{zz} \big).\label{eq:69}
 \end{equation}
  One can check easily that if we have numbers $\alpha_i$, $\beta_i>0$, $\gamma_i>0$ ($i=1, 2$) such that $\alpha_i^2\leq\beta_i\gamma_i$, then one has $(a_1\alpha_1+a_2\alpha_2)^2\leq(a_1\beta_1+a_2\beta_2)(a_1\gamma_1+a_2\gamma_2)$. Applying to $U_{\rho z}$, $-U_{zz}$ and $\frac{U_\rho}\rho + U_{zz}$ we obtain the result.
\end{proof}

We can now describe all the solutions:
\begin{cor}\label{cor:f-U}
  Suppose that $f$ is a convex piecewise linear positive function on $\Bbb{R}$, with slopes $-1=f'_0<f'_1<\cdots<f'_r=1$, and angular points at $z_1<\cdots<z_r$. It follows that
  \begin{equation}
    \label{eq:70}
    f(z) = A + \sum_{i = 1} ^r a_i |z-z_i|
  \end{equation}
  for real numbers $a_i=\frac12(f'_i-f'_{i-1}) > 0$ satisfing $\sum_1^r a_i=1$.
  
Then the corresponding harmonic function
\begin{equation}
 U = A \log \rho^2 + \sum_{i = 1} ^r a_i U_0(\rho,z-z_i) \label{eq:71}
\end{equation}
satisfies the conditions $V>0$ and $U_{\rho z}^2+U_{zz}^2>0$ of corollary \ref{cor:harmarck} if and only if $A>0$.
\end{cor}
\begin{proof}
  If the condition $A>0$ is satisfied, then $U(\rho,z)$ is a barycenter of functions $U_{n_i}(\rho,z-z_i)$ with $n_i>0$ and the result follows from proposition \ref{prop:superposition}.

  Conversely, if $A\leq0$ one can show that the inequality (\ref{eq:69}) is not true. : one calculates easily all the terms of the inequality, and making $\rho\rightarrow0$ it would imply an inequality on the function $f(z)$ which is satisfied only for $A>0$. We leave the details of the calculation to the reader, since we will not really use it in our constructions.
\end{proof}

\subsection{Solutions}
\label{sec:solutions}

We can now summarize what we have done. Any extremal Kähler Bach flat metric on a toric manifold with polytope $P$ and scalar curvature vanishing on the edge at infinity, is generated via the ansatz of corollary \ref{cor:harmarck} from an axisymmetric harmonic function $U(\rho,z)$ with the boundary conditions:
\begin{equation}\label{eq:72}
  U \sim
  \begin{cases}
    f(z) \log \rho^2, & \rho \rightarrow 0, \\
    U_0(\rho,z), & R\rightarrow \infty.
  \end{cases}
\end{equation}
Here $f:\Bbb{R}\rightarrow \Bbb{R}_+^*$ is a convex piecewise linear function, with slopes $\pm1$ at $\pm\infty$, and the first equivalent is valid only locally with respect to $z$. Conversely:

\begin{prop}\label{prop:f-U}
  Given any convex, piecewise linear function $f:\Bbb{R}\rightarrow\Bbb{R}_+^*$ with slopes $\pm1$ at $\pm\infty$, and satisfying the condition $A>0$ when it is written under the form (\ref{eq:70}), there is an axisymmetric harmonic function $U$ on $\Bbb{R}^3$ satisfying (\ref{eq:72}), which is unique up to the addition of an affine function of $z$. This function generates a Hermitian toric Ricci flat ALF metric via corollary \ref{cor:harmarck} on $\{\rho>0\}$.
\end{prop}
We say above $\rho>0$ because for a general $f$ as in the proposition, there is no reason to produce a metric which can be compactified into a smooth metric. This will be studied in the next section.

Observe that by remark \ref{rem:homogeneity}, if the function $f(z)$ satisfies the hypothesis of the proposition, then for $a\neq0$ the function $|a|^{-1}f(az+b)$ also does, and the two metrics generated by $f(z)$ and $|a|^{-1}f(az+b)$ are homothetic.
\begin{proof}
  We have already proved most of the proposition, there remains to prove only the uniqueness of $U$ satisfying the asymptotics (\ref{eq:72}). If we have two such axisymmetric harmonic functions $U_1$, $U_2$, then $U_1(\rho,z)-U_2(\rho,z)=o(\log \rho)$ when $\rho\rightarrow0$. By \cite[Theorem 6.4]{HarPol70} it follows that $U_1-U_2$ extends smoothly across $\rho=0$. Therefore $U_1-U_2$ is a global harmonic function on $\Bbb{R}^3$. From the behaviour when $R\rightarrow+\infty$ we deduce that $U_1-U_2=O(R \log R)$ which implies that $U_1-U_2$ is an affine function on $\Bbb{R}^3$. Since it is axisymmetric, it is an affine function of $z$ alone.
\end{proof}

\section{The moment map}
\label{sec:moment-map}

We now recover the polytope $P$ from the generating function $U$, and so ultimately from the convex piecewise linear function $f$ of proposition \ref{prop:f-U}.
We start from a Hermitian Ricci flat ALF metric generated by a function $U$ as in corollary \ref{cor:harmarck}. By construction we also have the Kähler metric $g_K=x_1^2g$. The Hamiltonian Killing vector fields $\partial_t$ and $\partial_{x_3}$ have moments $x_1$ and $\mu$ with respect to $\omega_K$. These can be calculated from $U$ by:
\begin{prop}
  As functions of $\rho$, $z$, the moments $x_1$, $\mu$ are given by
  \begin{equation}
    \label{eq:73}
    x_1 = \frac2{H_z}, \qquad \mu=-\frac2k\big( z + \frac{\rho H_\rho-2H}{H_z} \big).
  \end{equation}
\end{prop}
\begin{proof}
  By (\ref{eq:36}) and (\ref{eq:45}) we have $x_1=\frac2{\rho U_\rho}=\frac2{H_z}$. To obtain $\mu$ we calculate the Kähler form $\omega_K$: from (\ref{eq:18}) we deduce
  \begin{equation}
    \label{eq:74}
    \omega_K = - \frac{d\xi}{\xi^2} \wedge (dt-F dx_3) + \frac{Ve^u}{\xi^2} dx_2 \wedge dx_3.
  \end{equation}
  Using (\ref{eq:35}) and (\ref{eq:36}) it follows that
  \begin{equation}
    \label{eq:75}
    \omega_K = \frac2{U_\rho^2} \left( \frac1\rho(U_{zz}d\rho-U_{\rho z}dz)\wedge(dt-Fdx_3)
      - V (U_{\rho z}d\rho+U_{zz}dz)\wedge dx_3 \right).
  \end{equation}
  Calculating with (\ref{eq:43})--(\ref{eq:45}) we obtain that $d\mu = - \partial_{x_3} \lrcorner \omega_K$ is given by the formula
  \begin{multline}\label{eq:76}
    d\mu = \frac2{kH_z^2} \Big(
      \big((\rho H_\rho-2H) H_{\rho z} + \rho H_z H_{zz}\big) d\rho \\
      + \big(H_z^2 + (\rho H_\rho-2H) H_{zz} - \rho H_z H_{\rho z}\big) dz \Big)
  \end{multline}
and the formula for $\mu$ follows.
\end{proof}

The image of the moment map is the polytope $P$. The (finite) boundary is obtained for $\rho=0$, while the boundary at infinity is obtained for $R=\infty$, that is $x_1=0$. Suppose that $f$ has $r+1$ different slopes $f'_0=-1<f'_1<\cdots<f'_r=1$ and denote $z_1<...<z_r$ the angular points of $f$, so that the slope of $f$ on the segment $[z_i,z_{i+1}]$ is $f'_i$ (we take $z_0=-\infty$ and $z_{r+1}=+\infty$).
\begin{prop}\label{prop:moment-f}
  Suppose $U$ is generated from a convex piecewise linear function $f(z)$ as in proposition \ref{prop:f-U}. Then for $i=0,...,r$:
  \begin{itemize}
  \item On each segment $(z_i,z_{i+1})$ on which $f$ is affine and non constant:
    \begin{itemize}
    \item the function $F(0,z)$ is constant; we denote this constant $F_i$;
    \item one has
      \begin{equation}
        \label{eq:77}
        x_1 = \frac1{f(z)}, \qquad \mu = - \frac{F_i}{f(z)} + \frac2k \big( \frac{f(z)}{f'(z)} - z \big) .
      \end{equation}
      In particular $\mu+F_ix_1$ is constant on each segment of $f$, and it follows that a normal vector to $\partial P$ along the corresponding edge is $\partial_{x_3}+F_i \partial_t$.
    \end{itemize}
  \item On a segment $[z_i,z_{i+1}]$ where $f$ is constant, then $x_1=\frac1{f(z)}$ is constant and a normal vector to $\partial P$ is $\partial_t$.
  \end{itemize}
\end{prop}
\begin{proof}
  We calculate the asymptotics of $F$ near $\rho=0$. In the calculations below we take derivatives of equivalents, but all our assertions can be checked easily since $U$ is given by the explicit formula (\ref{eq:71}). So when $\rho\rightarrow0$ we have $U(\rho,z)\sim f(z)\log \rho^2$ and it follows that on a segment $(z_i,z_{i+1})$ where $f'(z)\neq0$:
  \begin{equation}
    \label{eq:78}
      \rho U_\rho \sim 2f(z), \quad U_z \sim f'(z)\log \rho^2, \quad
      U_{\rho z}\sim \frac{2f'(z)}\rho.
    \end{equation}
    One also calculates
    \begin{equation}
      \label{eq:79}
      U_{zz} = - 2 \sum_{i = 1} ^r \frac{a_i}{\sqrt{\rho^2+(z-z_i)^2}}
    \end{equation}
    which in particular extends on $\rho=0$ in the interior of the interval $(z_i,z_{i+1})$. It follows that
    \begin{equation}
      \label{eq:80}
      U_{\rho z}^2+U_{zz}^2 \sim \frac{4f'(z)^2}{\rho^2}.
    \end{equation}
From (\ref{eq:44}) we then obtain that $F$ also extends on $\rho=0$ with
\begin{equation}
  \label{eq:81}
  F(0,z) = \frac2k \big( \frac{f(z)^2}{f'(z)} - H(0,z) \big).
\end{equation}
Since $H_z=\rho U_\rho= 2f(z)$ on $\rho=0$, we obtain $F(0,z)_z=0$ so $F(0,z)$ is constant on $(z_i,z_{i+1})$.

The formula (\ref{eq:77}) is then an immediate consequence of (\ref{eq:73}) and (\ref{eq:45}). The last assertion on a segment $(z_i,z_{i+1})$ where $f$ is constant is tautological.
\end{proof}

\begin{cor}
  Under the same hypotheses the moment polytope $P$ of $g_K$ has $r+1$ edges $E_i$, parametrized by the segments $[z_i,z_{i+1}]$ ($i=0...r$) via the formula (\ref{eq:77}), and one additional edge at infinity $E_\infty$ with equation $x_1=0$.
\end{cor}

\section{Regularity}
\label{sec:regularity}

We now study the compactification of a metric coming from a generating function $U$. Since we have a toric Kähler metric, we can apply the toric theory: the compactification is a smooth manifold if the polytope satisfies the Delzant condition, that is two consecutive (finite) edges have primitive integral normals which form an integral basis of a lattice $\Bbb{Z}^2\subset \Bbb{t}^2$. The point here is that the integral structure, that is the lattice $\Bbb{Z}^2$, is not a priori given, since the Tod ansatz is local. In particular there is no reason why the coordinates $(t,x_3)$ above would be nice angular coordinates: they parametrize $\Bbb{t}^2$ but we have to find the lattice inside.

We will treat the general case where we allow the metric $g$ to have conical singularities along the divisor $D_i$ which is the preimage by the moment map of the finite edge $E_i$ of the polytope, say of angle $2\pi\alpha_i$  where $\alpha_i>0$. Each $D_i$ is a 2-sphere except for the two edges going to infinity where the point at infinity is missing.

\subsection{The regularity criterion}
\label{sec:regularity-criterion}

We denote $f_i=f(z_i)$.
\begin{lem}
  The metric $g_K$ extends over the sphere $D_i$ (outside the two fixed points) with a conical singularity of angle $2\pi\alpha_i$ if the vector $v_i$, normal to the edge $E_i$,defined by
  \begin{equation}
    \label{eq:82}
    v_i =
    \begin{cases}
      \alpha_i f'_i (\partial_{x_3}+F_i\partial_t) & \text{ if }f'_i \neq 0, \\ \frac2k \alpha_i f_i^2  \partial_t & \text{ if }f'_i = 0,
    \end{cases}
  \end{equation}
  is primitive in $\Bbb{Z}^2$.
\end{lem}
\begin{proof}
  Suppose $f'_i\neq0$. The formulas (\ref{eq:78})--(\ref{eq:79}) imply that $V$ extends over $\rho=0$ along the segment $(z_i,z_{i+1})$ with
  \begin{equation}
    \label{eq:83}
    V(0,z) = - \frac1k \big( 2 f(z) + \frac{f(z)^2}{f'(z)^2} U_{zz} \big)
  \end{equation}
  and similarly
  \begin{equation}
    \label{eq:84}
    e^{2\nu(0,z)} = f'(z)^2 V(0,z).
  \end{equation}
  From formula (\ref{eq:41}) we see that $g$ (or $g_K=x_1^2g$) can be compactified over the segment $(z_i,z_{i+1})$ if the vector field $\partial_{x_3}+F_i\partial_t\in \ker(dt-F_i dx_3)$ generates a circle, which will be contracted into a point at $\rho=0$. Suppose this is the case, so we have an integral generator $\lambda(\partial_{x_3}+F_i\partial_t)$ of the circle. Then the metric extends continuously with conical singularity of angle $2\pi\alpha_i$ if $e^{-\nu}\sqrt V dx_3(\lambda(\partial_{x_3}+F_i\partial_t))=\alpha_i$, that is $\pm \frac \lambda{f'_i}=\alpha_i$. We can take $\lambda=\alpha_if'_i$, and so we need $\alpha_if'_i(\partial_{x_3}+F_i\partial_t)$ to be an (integral) generator of a circle. If this is the case, then the standard regularity theory for toric extremal Kähler metrics tells us that the metric extends over the divisor $D_i$ (in the conical case, this means that the symplectic potential has the behaviour described in section \ref{sec:toric-kahl-geom}).

  The case where $f'_i=0$ is similar but the behaviour of the coefficients of $g$ is different. One calculates that $U_{\rho z}=O(\rho)$ and it follows that
  \begin{equation}
    \label{eq:85}
    V \sim - \frac4k \frac{f_i^2}{\rho^2U_{zz}}.
  \end{equation}
Therefore the contracted direction is now generated by $\partial_t$. We need the integral generator $\lambda\partial_t$ of the circle to satisfy $e^{-\nu}V^{-\frac12}\lambda\sim \alpha_i\rho$, and this gives $\lambda=\frac2k f_i^2\alpha_i$.
\end{proof}

Since $F$ is defined only up to an additive constant, we can always choose $F_0=0$. In the case $\alpha_0=1$ this implies that $\partial_{x_3}$ is primitive and therefore $\mu$ is primitive and gives the normal of the first edge $E_0$ of the polytope, corresponding to the segment $(-\infty,z_1)$ of $f$. In the general case $\alpha_0\neq 1$, that is when we have conical singularities, the function $\mu$ is still the weighted equation $\frac1{\alpha_0} \lambda_{E_0}$ used to describe the symplectic potential in (\ref{eq:57}). In all cases from (\ref{eq:77}) we see that the value of $\mu$ at the two fixed points of $P$ at infinity is $\frac2k \lim_{z\rightarrow \pm \infty} (\frac{f(z)}{f'(z)}-z)$. In terms of formula (\ref{eq:71}) this is $\frac2k (\pm A-\sum_1^r a_iz_i)$. From the normalization (\ref{eq:57}) on the polytope, the length of the edge at infinity with respect to $\mu$ should be $2$, which means that we want $\frac{4A}k=2$, that is
\begin{equation}
  \label{eq:86}
  k = 2A.
\end{equation}
In particular this gives the value $k=4n$ used for the Taub-NUT metric in section \ref{sec:kerr-taub-bolt}.
\begin{rem} Formula (\ref{eq:86}) can be alternatively derived as follows. From the ALF condition (\ref{eq:1}), we infer that $\frac{\partial ^2 \psi}{\partial x_1 \partial x_1} \sim \frac{1}{x_1 ^2}$ at infinity, see (\ref{eq:63}), while $\frac{\partial ^2 \psi}{\partial x_1 \partial x_1} = \frac{\partial \nu _1}{\partial x_1} = W = \frac{V}{x _1 ^2}$, cf. Remark  \ref{rem:def-v}. It follows that $V \sim 1$ at infinity, while from  (\ref{eq:71})  and  (\ref{eq:43}) we easily infer  $V \sim \frac{2 A}{k}$ when $R$ tends to infinity.
\end{rem}

There remains to study the regularity at the fixed points.
We apply the Delzant condition to deduce that the metric $g_K$ (hence $g$) extends to a smooth compactification $X^4$ as a metric with conical singularities if for $i=1,...,r$ the basis $(v_{i-1},v_i)$ is an integral basis of $\Bbb{Z}^2$. We can compare $v_{i-1}$ and $v_i$ thanks to:
\begin{prop}\label{prop:regularity}
  If $f'_{i-1}\neq0$ and $f'_i\neq0$, then
  \begin{align}
    \label{eq:87}
    F_i - F_{i-1} &= \frac2k f_i^2 \big( \frac1{f'_i} - \frac1{f'_{i-1}} \big). \\
\intertext{If $f'_i=0$, then}
    \label{eq:88}
    F_{i+1}-F_{i-1} &= \frac2k \left( f_i^2 \big( \frac1{f'_{i+1}} - \frac1{f'_{i-1}} \big) - 2 (z_{i+1}-z_i) f_i \right).
  \end{align}
\end{prop}
\begin{proof}
  The first formula is an immediate consequence of (\ref{eq:81}), because $H(0,z)$ is continuous. The second formula is also a consequence, because $H(0,z)_z=2f(z)=2f_i$ on $(z_i,z_{i+1})$, so the variation of $H(0,z)$ on this interval is $2(z_{i+1}-z_i)f_i$.
\end{proof}
\begin{rem}
  Using this proposition, formula (\ref{eq:82}) for the normals and the convexity of the polytope, one can give another proof of the convexity of $f(z)$ (corollary \ref{cor:f-convex}).
\end{rem}

This gives our final classification result:
\begin{thm}\label{thm:classification}
  Suppose that we have a convex piecewise linear function $f(z)$ with slopes $-1=f'_0<\cdots<f'_r=1$, singular points $z_1<\cdots<z_r$, and satisfying the condition $A>0$ of corollary \ref{cor:f-U}. Define the constants $F_i$ for all $i$'s such that $f'_i\neq0$ by $F_0=0$ and equations (\ref{eq:87})--(\ref{eq:88}). Then the corresponding instanton extends to a smooth compactification $X^4$ and the metric $g$ has conical singularities of angles $2\pi\alpha_0$,..., $2\pi\alpha_r$ along the divisors $D_i$ if and only if the normals $v_i$ to the edge $E_i$ of the polytope defined by (\ref{eq:82}) satisfy the Delzant condition, that is each pair $(v_{i-1},v_i)$ can be deduced from $(v_0,v_1)$ by a transformation in $GL(2,\Bbb{Z})$.

  Conversely, all toric Hermitian Ricci flat ALF metrics, with conical singularities around the fixed point loci, are obtained by this construction.
\end{thm}
Note that this theorem is really constructive: given $f$, Tod's generating function $U$ is given by (\ref{eq:71}) and then the metric by (\ref{eq:41}).
\begin{proof}
  The only fact which maybe remains to be proved is that the Ricci flat metric $g$ is indeed ALF. One can see this from the Kähler formalism along the lines of section \ref{sec:case-edge-at}. Note in particular that since we have explicit formulas for $U$ we have a much better control than the one stated there.

  Alternatively we can also see directly the ALF nature of $g$ from the formulas (\ref{eq:41})--(\ref{eq:45}), since we know that $U(\rho,z) \sim U_0(\rho,z)$ given by (\ref{eq:61}) when $R\rightarrow\infty$ (and again $U$ is actually explicit so we have a full asymptotic development). It is the same asymptotic behaviour for all these metrics, and in particular coincides with that of the Taub-NUT potential, and we calculated the corresponding metric (\ref{eq:67}). We deduce that we have
  \begin{equation}
    \label{eq:89}
    g \sim d\rho^2 + dz^2 + \rho^2 dx_3^2 + dt^2 = dr^2 + r^2 (d\theta^2+\sin^2(\theta)dx_3^2)+dt^2 
  \end{equation}
  which is the common asymptotic behaviour of all ALF metrics.
\end{proof}

The condition on the function $f$ in order to get a nice compactification is very strong. As we will see, it can be used in practice to understand the solutions.

\begin{rem}
  It is not surprising that the Delzant condition on the normal vectors $(v_i)$ is the same as the condition used in the physical literature to understand the rod structures which give rise to smooth solutions, see for example \cite{CheTeo10}. The data of the Delzant polytope is close but not identical to that of the rod structure.
\end{rem}

\subsection{The basic examples ($r\leq 2$)}
\label{sec:basic-examples-nleq}

We have already seen the Taub-NUT example in section \ref{sec:kerr-taub-bolt}. We now describe all the examples where the polytope has three finite edges (and one edge at infinity). We recover in this way the family of the Kerr-Taub-bolt metrics. We only derive from the possible functions $f(z)$ all the possible solutions, but we leave to the reader the explicit calculation of the polytopes from the usual formulas for the Kerr-Taub-bolt metrics.

The piecewise linear function $f(z)$ (figure \ref{fig:ktb-f}) has two singular points that we choose to be $z_1=-b$ and $z_2=b$. We write $f_1=f(z_1)=b+a+m+n$ and $f_2=f(z_2)=b-a+m+n$, where the parameters $a$, $b$, $m$ and $n$ are fixed by the condition $b^2=a^2+m^2-n^2$ and $|n|\leq m$. The convention is consistent with the usual parameters of the Kerr-Taub-bolt metrics.
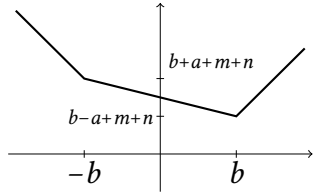
\begin{figure}[h]
  \centering
  \begin{tikzpicture}
    \draw[very thin, ->] (-2,0) -- (2,0);
    \draw[very thin, ->] (0,-0.5) -- (0,2);
    \draw[thick] (-1.9,1.9) -- (-1,1) -- (1,0.5) -- (1.9,1.4);
    \draw[very thin] (1,-0.05) -- (1,0.05) node[below]{$b$};
    \draw[very thin] (-1,-0.05) -- (-1,0.05) node[below]{$-b$};
    \draw[very thin] (-0.05,1) node[above right]{$\scriptstyle b+a+m+n$} -- (0.05,1);
    \draw[very thin] (-0.05,0.5) -- (0.05,0.5) node[left]{$\scriptstyle b-a+m+n$};
  \end{tikzpicture}
  \caption{The function $f(z)$ for the Kerr-Taub-bolt metrics}
  \label{fig:ktb-f}
\end{figure}
We have a slope $f'_1=- \frac ab$ on the middle segment. Note that the sign of $a$ is not fixed, and figure \ref{fig:ktb-f} corresponds to the case $a>0$. Let us begin by the case $a\neq0$. One obtains
\begin{equation}
  \label{eq:90}
  f(z) = m+n + \frac{b-a}{2b} |z+b| + \frac{b+a}{2b} |z-b|,
\end{equation}
in particular $k=2(m+n)$. We have the vectors $v_0=-(\partial_{x_3}+F_0\partial_t)$, $v_1=-\frac ab(\partial_{x_3}+F_1\partial_t)$ and $v_2=\partial_{x_3}+F_2\partial_t$. We are looking for metrics with cone angles $\alpha_0$, $\alpha_1$ and $\alpha_2$. Therefore we will get a smooth manifold if for some $\ell\in\Bbb{Z}$ one has
\begin{equation}
  \label{eq:91}
  \alpha_0v_0 \pm \alpha_2v_2 = \ell \alpha_1 v_1.
\end{equation}
We have the flexibility to vary $\alpha_1\in\Bbb{R}_+^*$, so we can actually suppose $\ell\in\{-1,0,1\}$ (the other values of $\ell$ correspond to quotienting by $\Bbb{Z}_{|\ell|}$). The equation (\ref{eq:91}) gives the system
\begin{equation}
  \label{eq:92}
  \begin{split}
    -\alpha_0 \pm \alpha_2 &= \ell \alpha_1 f'_1, \\
    -\alpha_0F_0 \pm \alpha_2F_2 &= \ell \alpha_1 f'_1 F_1.
  \end{split}
\end{equation}
The second equation, together with the first one and (\ref{eq:87}), gives
\begin{equation}
  \label{eq:93}
  \alpha_0 f_1^2 (1+\frac1{f'_1}) \pm \alpha_2 f_2^2 (1-\frac1{f'_1}) = 0.
\end{equation}
This implies that $\pm$ is actually a $+$. Taking $F_0=0$ (since $F$ is defined up to a constant), one calculates $F_1=-\frac1a((b+m)^2-(a+n)^2)$. We have two linear equations on $\alpha_0$, $\alpha_2$, and if $n\neq 0$ we calculate the solution:
\begin{equation}
  \label{eq:94}
  \alpha_0 = \frac \ell{4nb}\big( (b+m)^2-(a-n)^2 \big), \quad
  \alpha_2 = \frac \ell{4nb}\big( (b+m)^2-(a+n)^2 \big).
\end{equation}
In order to have positive angles, we need to take $\ell=\sign n$, hence the final formula
\begin{equation}
\label{eq:95}
  \alpha_0 = \frac 1{4|n|b}\big( (b+m)^2-(a-n)^2 \big), \quad
  \alpha_2 = \frac 1{4|n|b}\big( (b+m)^2-(a+n)^2 \big).
\end{equation}
So we see that for any such function $f$, there is a choice of angles so that the metric extends to a smooth manifold, with the corresponding conical singularities. The polytope is now calculated from proposition \ref{prop:moment-f}: we choose the integral basis of vectors $(v_1=-\frac ab(\partial_{x_3}+F_1\partial_t),-\alpha_0v_0=\alpha_0\partial_{x_3})$ which give us coordinates
\begin{equation}
  \label{eq:96}
  x = - \frac ab (\frac y{\alpha_0} + F_1x_1), \qquad y = \alpha_0 (\mu + \frac a{m+n}). 
\end{equation}
The translation on $y$ is here to center the polytope, since from (\ref{eq:77}) one obtains $\lim_{z\rightarrow\pm\infty}\mu(z)=\pm1-\frac a{m+n}$.

We now calculate the vertices $V_i=(x(z_i),y(z_i))$ of the polytope for $0\leq i\leq3$, where $z_0=-\infty$ and $z_3=+\infty$. We use formula (\ref{eq:77}). We have
\begin{equation}
V_0=(\frac ab,-\alpha_0), \qquad V_3 = \big( -\frac ab, \alpha_0 \big). \label{eq:97}
\end{equation}
Then $V_1$ has the same $y$ as $V_0$ (since $\partial_y$ is a normal to $E_0$), with $x_1(z_1)=\frac1{f(z_1)}=\frac1{b+a+m+n}$. Therefore from $x(z_1)=-\frac ab(-1+F_1x_1(z_1))$ one calculates
\begin{equation}
  \label{eq:98}
  V_1 = \big( \frac{b+m-n}b , -\alpha_0 \big).
\end{equation}
Similarly $V_2$ has the same $x$ coordinate as $V_1$ (since $\partial_x$ is a normal to $E_2$) and one obtains
\begin{equation}
  \label{eq:99}
  V_2 = \big( \frac{b+m-n}b , \alpha_0 \frac{b+m-a-3n}{b+m-a+n} \big) .
\end{equation}
A normal to the last edge $E_3$ is $v_2=-v_0+\ell v_1=-v_0+\sign(n) v_1$, so the form of the polytope depends on the sign of $n$. The polytopes are represented on figure \ref{fig:polytopesKTb}, for the same values of the parameters, except for the sign of $n$ (note that this exchanges $\alpha_0$ and $\alpha_2$), with the angles for each edge.
    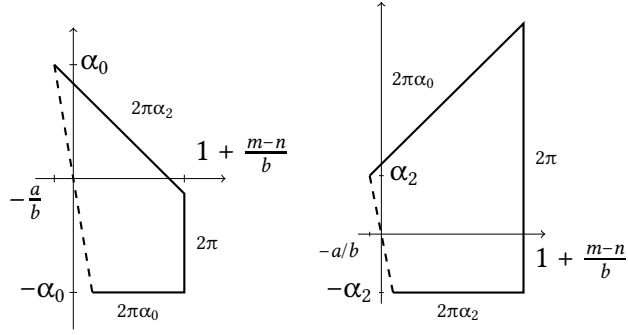
\begin{figure}[h]
      \centering
      \begin{tikzpicture}
        \draw[very thin,->] (-0.5,0) -- (2,0) ;
        \draw[very thin,->] (0,-2) -- (0,2) ;
        \draw[very thin] (-0.05,1.51) node[right]{$\alpha_0$} -- (0.05,1.51) ;
        \draw[very thin] (-0.05,-1.51) -- (0.05,-1.51)  node[left]{$-\alpha_0$};
        \draw[very thin] (-0.25,-0.05) -- (-0.25,0.05) node[below left]{$-\frac ab$};
        \draw[very thin] (1.46,0.05) -- (1.46,-0.05) node[above right]{$1+\frac{m-n}b$} ;
        \draw[thick,dashed] (-0.25,1.51) -- (0.25,-1.51);
        \draw[thick] (0.25,-1.51) -- node[below]{$\scriptstyle 2\pi\alpha_0$} (1.46,-1.51) -- node[right]{$\scriptstyle 2\pi$} (1.46,-0.2) -- node[above right]{$\scriptstyle 2\pi\alpha_2$} (-0.25,1.51);
      \end{tikzpicture}
      \begin{tikzpicture}[scale=0.615]
        \draw[very thin,->] (-0.5,0) -- (3.5,0) ;
        \draw[very thin,->] (0,-1.5) -- (0,5) ;
        \draw[very thin] (-0.05,1.26) node[right]{$\alpha_2$} -- (0.05,1.26) ;
        \draw[very thin] (-0.05,-1.26) -- (0.05,-1.26)  node[left]{$-\alpha_2$};
        \draw[very thin] (-0.25,-0.05) -- (-0.25,0.05) node[below left]{$\scriptstyle -a/b$};
        \draw[very thin] (3.04,0) node[below right]{$1+\frac{m-n}b$} ;
        \draw[thick,dashed] (-0.25,1.26) -- (0.25,-1.26);
        \draw[thick] (0.25,-1.26) -- node[below]{$\scriptstyle 2\pi\alpha_2$} (3.04,-1.26) -- node[right]{$\scriptstyle 2\pi$} (3.04,4.54) -- node[above left]{$\scriptstyle 2\pi\alpha_0$} (-0.25,1.26);
      \end{tikzpicture}
      \caption{Kerr-Taub-bolt polytopes with the two orientations: $n>0$ and $n<0$}
      \label{fig:polytopesKTb}
    \end{figure}

\subsubsection*{The case $a=0$, $n\neq 0$.} It is obtained at the limit $a\rightarrow0$. In that case $\alpha_0=\alpha_2$, so by renormalizing the angular variables by $\frac1{\alpha_0}$ we can interpret the solutions as being smooth along $E_0$ and $E_2$, and with conical singularity of angle $\frac{2\pi}{\alpha_0}$ along the 2-sphere corresponding to $E_1$. Writing $m=\lambda|n|$ we obtain $b=\sqrt{\lambda^2-1}|n|$ and an angle
\begin{equation}
  \label{eq:100}
  \frac{2\pi}{\alpha_0} = \frac{4\pi}{\sqrt{\lambda^2-1}+\lambda}
\end{equation}
which varies from $4\pi$ to $0$ when $\lambda$ varies from $1$ to $+\infty$, and takes the value $2\pi$ for $m=\frac54|n|$: this is the smooth Taub-bolt metric (with both orientations, depending on the sign of $n$). The limit $\lambda=1$ (with $b=0$ so the edge $E_1$ is removed) is the Taub-NUT metric, with the corresponding $S^2$ contracted to a point. One can check that a suitable rescaling when $\lambda\rightarrow1$ converges to a ramified double cover of the Eguchi-Hanson metric (hence the limit angle $4\pi$ around the 2-sphere). So one can consider the metrics for $\lambda>1$ small as obtained by blowing up the Taub-NUT metric at the fixed point and grafting a ramified double cover of the Eguchi-Hanson metric. This point of view will be explored systematically in section \ref{sec:blowing-up-cone}.

    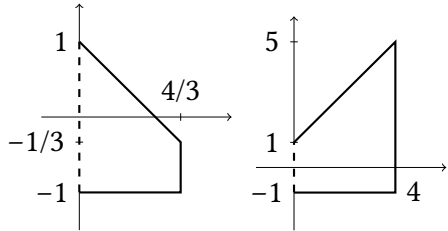
\begin{figure}[h]
      \centering
      \begin{tikzpicture}
        \draw[very thin,->] (-0.5,0) -- (2,0) ;
        \draw[very thin,->] (0,1) node[left]{$1$} -- (0,1.5) ; \draw[very thin] (0,-1.5) -- (0,-1) node[left]{$-1$};
        \draw[very thin] (1.333,-0.05) -- (1.333,0.05) node[above]{$4/3$} ;
        \draw[very thin] (-0.05,-0.333) node[left]{$-1/3$} -- (0.05,-0.333);
        \draw[thick,dashed] (0,-1) -- (0,1);
        \draw[thick] (0,-1) -- (1.333,-1) -- (1.333,-0.333) -- (0,1);
      \end{tikzpicture}
      \begin{tikzpicture}
        \draw[very thin,->] (-0.5,0) -- (2,0) ;
        \draw[very thin,->] (0,0.333) -- (0,2) ; \draw[very thin] (0,-0.833) -- (0,-0.333) node[left]{$-1$};
        \draw[very thin] (1.333,-0.05) node[below right]{$4$} -- (1.333,0.05) ;
        \draw[very thin] (-0.05,1.666) node[left]{$5$} -- (0.05,1.666);
        \draw[very thin] (-0.05,0.333) node[left]{$1$} -- (0.05,0.333);
        \draw[thick,dashed] (0,-0.333) -- (0,0.333);
        \draw[thick] (0,-0.333) -- (1.333,-0.333) -- (1.333,1.666) -- (0,0.333);
      \end{tikzpicture}
      \caption{Taub-bolt polytopes (two orientations)}
      \label{fig:polytopesTb}
    \end{figure}

\subsubsection*{The case $n=0$.} Via formula (\ref{eq:87}) this condition corresponds exactly to $F_2=F_0$, which gives $v_2=-v_0$ and therefore $\ell=0$. The corresponding solutions are smooth and we obtain the Kerr family. We let the reader check that we obtain the polytopes on figure \ref{fig:polytopesTNSK}. The parameter $p=\frac ab$ varies in the interval $(-1,1)$, with a symmetry between $p$ and $-p$. For $p=0$ one recovers the Schwarzschild metric; when $p\rightarrow 1$ the Kerr metric degenerates to the Taub-NUT metric.
    \begin{figure}[h]
      \centering
      \begin{tikzpicture}
        \draw[very thin,->] (-0.5,0) -- (2.5,0) ;
        \draw[very thin,->] (0,1) node[left]{$1$} -- (0,1.5) ; \draw[very thin] (0,-1.5) -- (0,-1) ;
        \draw[very thin,color=blue] (0.6,-0.05) node[below]{$p$} -- (0.6,0.05);
        \draw[thick,dashed,color=red] (0,-1) -- (0,1);
        \draw[thick,color=red] (0,-1) -- (2,-1) -- node[below right]{$2$} (2,1) -- (0,1);
        \draw[very thick,dashed] (1,1) -- (-1,-1);
        \draw[very thick] (-1,-1) -- node[below left]{$-1$} (1,-1) -- node[below right]{$1$} (1,1);
        \draw[dashed, thick,color=blue] (-0.6,-1) -- (0.6,1);
        \draw[thick,color=blue] (0.6,1) -- (1.8,1) --  (1.8,-1) node[below]{$\scriptstyle 1+\sqrt{1-p^2}$}-- (-0.6,-1);
      \end{tikzpicture}

      \caption{Taub-NUT, Schwarzschild (red) and Kerr (blue) polytopes}
      \label{fig:polytopesTNSK}
    \end{figure}
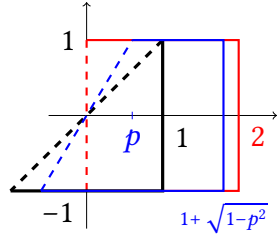

\section{The classification of smooth Hermitian ALF instantons}
\label{sec:class-smooth-herm}

\subsection{The constraints}
\label{sec:constraints}

Smoothness imposes strong constraints on the convex piecewise linear function $f$ from section \ref{sec:regularity-criterion}. Suppose we have three consecutive segments of $f$, say with slopes $f'_{j-1} < f'_j < f'_{j+1}$, then the smoothness at the vertex corresponding to $z_j$ is expressed by saying that the two basis $(v_{j-1},v_j)$ and $(v_j,v_{j+1})$ of vectors defined by (\ref{eq:82}) generate the same lattice, that is satisfy
\begin{equation}
  \label{eq:101}
  v_{j-1} + \varepsilon_j v_{j+1} = \ell_j v_j, \qquad \varepsilon_j = \pm 1, \ell_j \in \Bbb{Z}.
\end{equation}
We get the system
\begin{equation}
  \label{eq:102}
  \begin{split}
    f'_{j-1} + \varepsilon_j f'_{j+1} &= \ell_j f_j, \\
    f'_{j-1}F_{j-1} + \varepsilon_j f'_{j+1}F_{j+1} &= \ell_j f'_j F_j.
  \end{split}
\end{equation}
Using the first equation, the second equation becomes
\begin{equation}
  \label{eq:103}
  f'_{j-1} (F_j-F_{j-1}) = \varepsilon_j f'_{j+1} (F_{j+1}-F_j).
\end{equation}
Using (\ref{eq:87}) we obtain
\begin{equation}
  \label{eq:104}
  f_{j+1}^2 = \varepsilon_j f_j^2 \frac{f'_j-f'_{j-1}}{f'_{j+1}-f'_j}.
\end{equation}
Since $f'_{j-1}<f'_j<f'_{j+1}$ we deduce that $\varepsilon_j=+1$ and therefore the system (\ref{eq:102}) is now reduced to the equations
\begin{equation}
  \label{eq:105}
  \begin{split}
    f'_{j-1} + f'_{j+1} &= \ell_j f_j, \\
    f_{j+1}^2 &= f_j^2 \frac{f'_j-f'_{j-1}}{f'_{j+1}-f'_j}.
  \end{split}
\end{equation}
Using (\ref{eq:82}) in the case where $f'_{j-1}$ or $f'_{j+1}$ vanishes, the calculation leads to the same conclusion: $\varepsilon_j=+1$ and one has (\ref{eq:105}). If finally $f'_j=0$, then the calculation gives $\varepsilon_j=+1$ and
\begin{equation}
  \label{eq:106}
  \begin{split}
    f'_{j+1} &= - f'_{j-1}, \\
    z_{j+1}-z_j &= \frac{2+\ell_j}2 \frac{f_j}{f'_{j+1}}
  \end{split}
\end{equation}

\begin{lem}\label{lem:n3}
  Suppose that we have a smooth Hermitian ALF instanton generated by a convex piecewise linear function $f$ on $\Bbb{R}$ as before, with $r+1$ different slopes $-1=f'_0<f'_1<\cdots<f'_r=1$. Then in three successive slopes $f'_{j-1}<f'_j<f'_{j+1}$ we must have $f'_{j-1}<0$ and $f'_{j+1}>0$. In particular $r\leq3$.
\end{lem}
\begin{proof}
  Indeed suppose for example that $f'_{j-1}<f'_j<f'_{j+1}\leq0$. Then the first equation in  system (\ref{eq:105}) implies that $\ell_j\geq2$. But $f_{j+1}<f_j$ so the second equation (\ref{eq:105}) tells us that $f'_{j+1}-f'_j>f'_j-f'_{j-1}$, that is $f'_{j-1}+f'_{j+1}=\ell_jf'_j>2f'_j$ which is a contradiction.
\end{proof}

\subsection{The case $r=3$}
\label{sec:case-n=3}

Suppose now that we have a smooth instanton with $n=3$, so by lemma \ref{lem:n3} we have $-1=f'_0 < f'_1 < 0 < f'_2 < f'_3=1$. We will denote $p=-f'_1$ and $q=f'_2$ so that both are inside $(0,1)$. This is illustrated in figure \ref{fig:ct-f}.
\begin{figure}[h]
  \centering
  \begin{tikzpicture}
    \draw[very thin, ->] (-2,0) -- (2,0);
    \draw[very thin, ->] (0,-0.5) -- (0,2.5);
    \draw[thick] (-1.9,1.9) -- (-0.67,0.67) node[left]{$f_1$} -- node[above]{$\scriptstyle -p$} (0,0.5) node[right]{$f_2$} -- node[above]{$\scriptstyle q$} (1,1.1) node[right]{$f_3$} -- (1.9,2);
  \end{tikzpicture}
  \caption{The function $f(z)$ for the Chen-Teo metrics}
  \label{fig:ct-f}
\end{figure}
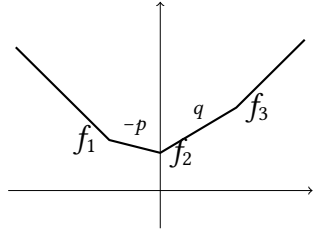
Applying twice (\ref{eq:105}) we obtain the equations
\begin{equation}\label{eq:107}
  \begin{split}
    -1+q = - \ell_1 p, \qquad  -p + 1 = \ell_2 q, \\
    (1-p) f_1^2 = (p+q) f_2^2 = (1-q) f_3^2 .
  \end{split}
\end{equation}
From the first equation we see that $\ell_1, \ell_2>0$. On the other hand, since $f_2<f_1$, by the second equation we have $1-p<p+q$ which gives $2p>1-q=\ell_1 p$. Finally we get $\ell_1=1$, and similarly $\ell_2=1$. In particular (\ref{eq:101}) now becomes $v_0+v_2=v_1$ and $v_1+v_3=v_2$, which imply $v_3=-v_0$, that is these are AF instantons.

Now the system (\ref{eq:107}) reduces to $p+q=1$ and
\begin{equation}
  \label{eq:108}
  f_1 = \frac{f_2}{\sqrt q}, \qquad f_3 = \frac{f_2}{\sqrt p}.
\end{equation}
This completely determines $f$ up to the transformation of $f$ into $\frac1a f(az+b)$. Therefore we have constructed a 1-dimensional family, indexed by $p\in(0,1)$, which is exactly the Chen-Teo family. We have proved:

\begin{thm}
  A smooth Hermitian, non hyperKähler, ALF gravitational instanton, is either:
  \begin{itemize}
  \item the Taub-NUT metric (with the orientation opposed to the hyperKähler orientation)
  \item the 1-parameter family of Kerr instantons
  \item the Taub-bolt metric (with respect to both orientations)
  \item the 1-parameter family of Chen-Teo instantons.
  \end{itemize}\qed
\end{thm}

Remark that this is a classification theorem, we do not need to compare our metrics with the explicit formulas for the Chen-Teo instantons: by uniqueness our 1-dimensional family has to be the Chen-Teo family. In particular we have given an independent construction of the Chen-Teo instantons. Nevertheless it is possible to compare directly our data (the function $f$) with the explicit formulas of Chen-Teo, but we do not need it here.

We can calculate the polytopes for the 1-parameter family of Chen-Teo instantons. We fix $p,q\in(0,1)$. We can choose $f_2=pq$ so that $f_1=p\sqrt q$ and $f_3=q\sqrt p$. Then we take $z_2=0$ and therefore $z_1=-\frac{f_1-f_2}p=q-\sqrt q$, similarly $z_3=\sqrt p-p$. We put $f$ in the form (\ref{eq:70}):
\begin{equation}
  \label{eq:109}
  \begin{split}
    &f(z) = \frac 12 \left( k + q | z + \sqrt q - q | + |z| + p | z - \sqrt p + p| \right), \\
    &\text{with }
    k = 1 - p^{\frac32}- q^{\frac32} .
  \end{split}
\end{equation}
We choose $F_0=0$ and therefore $F_1=\frac2k p^2q (-\frac1p+1)=-\frac2k pq^2$. An integral basis of the lattice is $(v_1=-p(\partial_{x_3}+F_1\partial_t),-v_0=\partial_{x_3})$ which gives coordinates
\begin{equation}
x=-p(y+F_1x_1),\quad y=\mu+\frac1k(p^{\frac32}-p^2-q^{\frac32}+q^2),\label{eq:110}
\end{equation}
where the constant added to $\mu$ in $y$ is chosen so that $y|_{E_\infty}$ varies in $(-1,1)$, see the definition of $k$ in (\ref{eq:86}). This is chosen so that the first edge $E_0$ is horizontal and the second one $E_1$ is vertical. The form of the polytope is now given by its normals to the edges: $v_0$, $v_1$, $v_2=v_1-v_0$, $v_3=v_2-v_1=-v_0$. Using the same method as in section \ref{sec:basic-examples-nleq}, one can calculate the vertices $V_i=(x(z_i),y(z_i))$ of the polytope (with $z_0=-\infty$ and $z_4=+\infty$). Skipping the details of the calculation, we get:
\begin{align}
    \notag V_0 &= (p,-1), & V_1 &= (p(1+\frac2k q^{\frac32}), -1),\\
    \label{eq:111} V_2&= \big(p(1+\frac2k q^{\frac32}), -1-\frac2k(q^{\frac32}-q)\big),\\
    V_3&=(-p(1-\frac2k \sqrt{p}q), 1), &
    \notag V_4&=(-p,1).
\end{align}
There is a symetry: the polytope for $(q,p)$ is obtained from that for $(p,q)$ by the integral transformation
\begin{equation}
(x,y)\mapsto (x+y,-y).\label{eq:112}
\end{equation}
This is somehow similar to the Kerr family with the symmetry between the parameter $a$ and $-a$, with $a=0$ begin the Schwarzschild metric. In our case, the polytope which admits a symmetry is obtained for $p=q=\frac12$, and has vertices
\begin{equation}
  \label{eq:113}
  (\frac12,-1), \quad (1+\frac1{\sqrt 2},-1), \quad (1+\frac1{\sqrt 2},0), \quad (\frac1{\sqrt 2},1), \quad (-\frac12,1).
\end{equation}
We represent on figure \ref{PolytopesCT} the symmetric polytope (in red), then a polytope in the family obtained for $p=0.1$ (in blue), and (dashed) the limit polytope when $p\rightarrow0$ which is the polytope of the Taub-bolt metric. Of course the picture for $p\geq\frac12$ is obtained by the symmetry (\ref{eq:112}).
  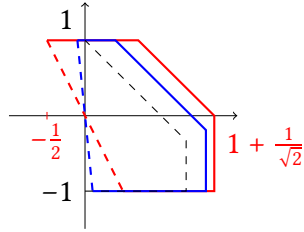
\begin{figure}[h]
    \centering
    \begin{tikzpicture}
      \draw[very thin,->] (-1,0) -- (2,0) ;
      \draw[very thin,->] (0,-1.5) -- (0,1.5) ;
      \draw[very thin,color=red] (-0.5,-0.05) -- (-0.5,0.05) node[below]{$-\frac12$} ;
      \draw (0,1)  node[above left]{$1$} ;
      \draw[thick,color=red] (0.5,-1) -- (1.7,-1) -- (1.7,0) node[below right]{$1+\tfrac1{\sqrt 2}$} -- (0.7,1) -- (-0.5,1) ;
      \draw[thick,dashed,color=red] (-0.5,1) -- (0.5,-1) ;
      \draw[thick,color=blue] (0.1,-1) -- (1.59,-1) -- (1.59,-0.19) -- (0.4,1) -- (-0.1,1) ;
      \draw[thick,dashed,color=blue] (-0.1,1) -- (0.1,-1) ;
      \draw[thin, dashed] (0,-1) node[left]{$-1$} -- (1.33,-1) -- (1.33,-0.33) -- (0,1);
    \end{tikzpicture}
    \caption{Polytopes of the Chen-Teo instantons}
    \label{PolytopesCT}
  \end{figure}

\begin{rem}
  From the data of a polytope, it is possible to calculate abstractly (without solving the extremal metric problem) the scalar curvature (we need it to vanish on the edge at infinity) and the Calabi functional of the extremal Kähler metric. Bach flat metrics appear as critical points of the Calabi functional when one varies the edges of the polytope without changing their normals, see \cite{CheLebWeb08}. (In our problem the edges come with a measure which is the integral measure multiplied by the angle $\alpha$; the Poincaré behaviour at infinity is obtained by taking a measure zero on the edge at infinity). So the polytopes which generate an instanton can a priori be characterized abstractly by algebraic equations, but these seem impossible to solve. Conversely it is possible to check that a given polytope, with given weights on the edges, can generate an instanton. We checked this on the Chen-Teo polytopes using Maple.
\end{rem}

\begin{rem}\label{rem:full-CT}
  We have not tried here to recover the whole 5-parameter family of Chen and Teo in \cite{CheTeo15}, which amounts for us to 4 parameters since we do not include scale. But one can remark that our piecewise linear functions with 3 singular points have exactly the same number of parameters: we can fix $f_2=f(0)$, then there remain 4 free parameters, the two slopes $p_1$, $p_2$ and the two values $f_1$, $f_3$. So it looks likely that this gives the full Chen-Teo family. Using our methods it is then possible to say which ones are conical (probably a 3-parameter subfamily, since the three cone angles give two additional parameters).
\end{rem}

\section{Blowing up and cone angles}
\label{sec:blowing-up-cone}

In this section we prove theorem \ref{thm:B} by describing a procedure which enables to construct new solutions from a given solution. Geometrically it is a blowing up of a fixed point of the torus, but we will see that there is a simple interpretation in terms of the piecewise linear function.

So suppose that we have a Hermitian, Ricci flat, ALF instanton, generated by a function $f(z)$ as before. We suppose that the instanton is defined on a smooth manifold but may have conical singularities along the fixed point sets. We shall describe a procedure to add a new segment to the function $f$. This is illustrated in figure \ref{fig:blowing_up}, where we want to add a segment with slope $p$ between two segments with slopes $p_j$ and $p_{j+1}$. At the end the procedure will be slightly different but it is useful to start with this idea.

\begin{figure}[h]
  \centering
  \begin{tikzpicture}
    \draw[very thin,->] (-2.5,0) -- (3.5,0) ;
    \draw (-2,1.5) node[left]{$f(z)$} -- (0,0.5) -- node[below]{$p_{j+1}$} (3,2) ;
    \draw (-1.5,1.25) node[below]{$p_j$} ; 
    \draw[very thick,dashed,color=blue] (-1,1) -- node[above]{$p$} (0.6,0.8) ;
    \draw[very thin] (-2,-0.05) -- (-2,0.05) node[below]{$z_j$} ;
    \draw[very thin] (0,-0.05) -- (0,0.05) node[below]{$z_{j+1}$} ;
    \draw[very thin] (3,-0.05) -- (3,0.05) node[below]{$z_{j+2}$} ;
  \end{tikzpicture}
  \caption{Blowing up}
  \label{fig:blowing_up}
\end{figure}
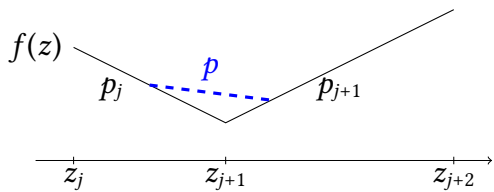

We denote $f_i=f(z_i)$, the slope of each segment is $p_i$ and a normal vector to the corresponding edge $E_i$ of the polytope is $v_i=p_i(\partial_{x_3}+F_i\partial_t)$ (for simplicity we suppose that all slopes are nonzero but the proof below can be adapted to this case). We have a conical singularity of angle $2\pi\alpha_i$ around the corresponding sphere, and from section \ref{sec:regularity-criterion} we know that since the manifold is smooth we have the relation $\alpha_{i-1}v_{i-1} + \alpha_{i+1}v_{i+1} = \ell_i \alpha_i v_i$ for some integers $\ell_i$ (the sign in the equation is $+$ because of the convexity of the polytope). We now fix a given $j$ and we want to add a new normal $v=p(\partial_{x_3}+F\partial_t)$  and an angle $\alpha$. So we have
\begin{equation}\label{eq:114} 
  \begin{split}
    \alpha_{j-1}v_{j-1} + \alpha_{j+1}v_{j+1} &= \ell_j \alpha_j v_j   \\
    \alpha_j v_j + \alpha_{j+2} v_{j+2} &= \ell_{j+1} \alpha_{j+1} v_{j+1} 
  \end{split}
\end{equation}
To keep the smoothness we want the new data to satisfy:
\begin{equation}\label{eq:115}
  \begin{split}
    \alpha_{j-1}v_{j-1} + \alpha v &= (\ell_j+1) \alpha_j v_j \\
    \alpha_j v_j + \alpha_{j+1} v_{j+1} &= \alpha v \\
    \alpha v + \alpha_{j+2} v_{j+2} &= (\ell_{j+1}+1) \alpha_{j+1}v_{j+1}
  \end{split}
\end{equation}
The choice of the integer constants in the RHSs is justified by the fact that the first equation in (\ref{eq:114}) follows from the two first equations in (\ref{eq:115}), and the second equation follows from the two last equations in (\ref{eq:115}). In particular, given the system (\ref{eq:114}) our more complicated system (\ref{eq:115}) is equivalent to the single middle equation, that is to the system
\begin{equation}
  \label{eq:116}
  \begin{split}
    \alpha_j p_j + \alpha_{j+1} p_{j+1} &= \alpha p, \\
    \alpha_j p_j (F_j-F) + \alpha_{j+1}p_{j+1} (F_{j+1}-F) &= 0.
  \end{split}
\end{equation}
We deduce the values of $p$ and $F$ in terms of $\alpha$:
\begin{equation}
  \label{eq:117}
  p = \frac{\alpha_j p_j + \alpha_{j+1} p_{j+1}}\alpha, \quad
  F = \frac{\alpha_jp_jF_j+\alpha_{j+1}p_{j+1} F_{j+1}}{\alpha_j p_j + \alpha_{j+1} p_{j+1}}.
\end{equation}
If $\alpha_j p_j + \alpha_{j+1} p_{j+1}=0$ then $p=0$ and we know that $F$ is not relevant. Here we suppose $p\neq 0$ (one can obtain the case $p=0$ either by using (\ref{eq:88}) or by taking a limit $p\rightarrow0$). Let us call $f_-$ and $f_+$ the values of $f$ at the boundary of the new segment. By (\ref{eq:87}) we must have
\begin{equation}
  \label{eq:118}
  F-F_j = \frac2k f_-^2 \big(\frac1p-\frac1{p_j}\big), \quad
  F_{j+1}-F = \frac2k f_+^2 \big(\frac1{p_{j+1}}-\frac1p\big).
\end{equation}
Writing $F_{j+1}-F_j=\frac2k f_{j+1}^2(\frac1{p_{j+1}}-\frac1{p_j})$, we finally deduce from (\ref{eq:117})
\begin{equation}
  \label{eq:119}
  f_-^2 = f_{j+1}^2 \frac{(p_{j+1}-p_j)\alpha_{j+1}}{\alpha_{j+1}p_{j+1}-(\alpha-\alpha_j)p_j}, \quad
  f_+^2 = f_{j+1}^2  \frac{(p_{j+1}-p_j)\alpha_j}{(\alpha-\alpha_{j+1})p_{j+1}-\alpha_jp_j}.
\end{equation}

Observe that for $\alpha=\alpha_j+\alpha_{j+1}$ the slope $p\in(p_j,p_{j+1})$ and $f_-=f_{j+1}=f_+$, so the procedure does not change anything (we added a segment of length $0$). We can rewrite (\ref{eq:119}) as
\begin{equation}\label{eq:120}
  f_{j+1}^2-f_-^2=f_{j+1}^2\frac{(\alpha_j+\alpha_{j+1}-\alpha)p_j}{\alpha(p-p_j)},\quad
  f_+^2-f_{j+1}^2=f_{j+1}^2\frac{(\alpha_j+\alpha_{j+1}-\alpha)p_{j+1}}{\alpha(p_{j+1}-p)}.
\end{equation}
Now suppose that $\alpha<\alpha_j+\alpha_{j+1}$ but is close to $\alpha_j+\alpha_{j+1}$. Then $f_{j+1}-f_-$ has the sign of $p_j$ and $f_+-f_{j+1}$ that of $p_{j+1}$. So $f_\pm=f(z_\pm)$ for some $z_-\in(z_j,z_{j+1})$ and $z_+\in(z_{j+1},z_{j+2})$, both close to $z_{j+1}$. So it looks possible to modify $f$ as in figure \ref{fig:blowing_up}, except that to introduce a segment of slope $p$ we would need $f_+-f_-=p(z_+-z_-)$ which does not follow from our formulas.

So we must do something slightly more complicated: we must modify globally the function $f$ in order to match the new segment of slope $p$. For $\alpha<\alpha_j+\alpha_{j+1}$ close enough to $\alpha_j+\alpha_{j+1}$ we have the slopes $p_0=-1< p_1< \cdots < p_j< p< p_{j+1}< \cdots < p_r=1$, the constants $F_0,...,F_j,F,F_{j+1},...$, the angles $\alpha_0,...,\alpha_j,\alpha,\alpha_{j+1},...$ and the corresponding values of the piecewise linear function $f_1,...,f_j,f_-,f_{j+1},f_+,f_{j+2},...$. This is coherent in the sense that the sign of each $f_{i+1}-f_i$ is the sign of $p_i$, as we have just checked for the new values $f_\pm$. So there exists a convex piecewise linear function $f_\alpha$ with these slopes and these values at the singular points ($f_\alpha$ is unique up to translation in the $z$ variable). From our construction the constraints of proposition \ref{prop:regularity} and the Delzant condition on the normals $\alpha_0v_0,...,\alpha_jv_j,\alpha v,\alpha_{j+1}v_{j+1},...$ are still satisfied, so we obtain a family of new instantons with the same normals and cone angles as the initial one, and one additional edge to the polytope with cone angle $2\pi\alpha$. Geometrically the addition of an edge to the polytope is a complex blowup. This finishes the proof of theorem \ref{thm:B}.

\bibliographystyle{alpha}
\bibliography{biblio,alf}

\end{document}